\documentclass[11pt]{amsart}
\title{The Genus-One Global Mirror Theorem for the Quintic Threefold}
\author{Shuai Guo}
\address{\newline School of Mathematics Science\newline Peking University
 \newline No 5. Yiheyuan Road\newline  Beijing 100871 China}
\email{guoshuai@math.pku.edu.cn}
\author{Dustin Ross}
\address{\newline Department of Mathematics\newline San Francisco State University\newline Thornton Hall 941\newline 1600 Holloway Avenue\newline San Francisco, California 94132 USA}
\email{rossd@sfsu.edu}

\usepackage {color,graphicx,psfrag, verbatim,amssymb,amscd,enumerate,subfigure}
\usepackage{texdraw,xypic,appendix,csquotes}
\input xy
\xyoption{all}
\usepackage[margin=1in]{geometry}

\newcommand{\bP}{\mathbb{P}}

\newcommand{\D}{\mathcal{D}}
\newcommand{\oD}{\overline{\mathcal{D}}}
\newcommand{\M}{\overline{\mathcal{M}}}

\newcommand{\cO}{\mathcal{O}}

\newcommand{\bC}{\mathbb{C}}

\newcommand{\B}{\mathcal{B}}
\renewcommand{\L}{\mathcal{L}}

\newcommand{\vir}{\mathrm{vir}}
\newcommand{\val}{\mathrm{val}}
\newcommand{\Aut}{\mathrm{Aut}}
\newcommand{\Contr}{\mathrm{Contr}}

\newcommand{\ev}{\mathrm{ev}}

\newcommand{\bt}{\mathbf{t}}

\renewcommand{\ev}{\mathrm{ev}}

\newcommand{\Res}{\mathrm{Res}}
\newcommand{\re}{\mathrm{e}}
\newcommand{\ri}{\mathrm{i}}

\newcommand{\diag}{\mathrm{diag}}
\newcommand{\loc}{\mathrm{loc}}

\newcommand{\I}{I}
\newcommand{\U}{\mathbb{U}}
\newcommand{\bq}{\mathbf{q}}
\newcommand{\bp}{\mathbf{p}}

\renewcommand{\cH}{\mathcal{H}}

\newcommand{\ft}{\mathfrak{t}}
\newcommand{\fq}{\mathfrak{q}}
\newcommand{\oF}{\overline{F}}
\newcommand{\oJ}{\overline{J}}
\newcommand{\oV}{\overline{V}}

\newtheorem{dummy}{}[section]
\newtheorem{lemma}[dummy]{Lemma}
\newtheorem{proposition}[dummy]{Proposition}
\newtheorem{theorem}[dummy]{Theorem}
\newtheorem{corollary}[dummy]{Corollary}
\newtheorem{conjecture}[dummy]{Conjecture}

\newtheorem*{mainresult}{Main Result}

\theoremstyle{definition}
\newtheorem{remark}[dummy]{Remark}

\usepackage{graphicx}

2
\begin{document}

\maketitle
\begin{abstract}
We prove the genus-one restriction of the all-genus Landau-Ginzburg/Calabi-Yau conjecture of Chiodo and Ruan, stated in terms of the geometric quantization of an explicit symplectomorphism determined by genus-zero invariants. This gives the first evidence supporting the higher-genus Landau-Ginzburg/Calabi-Yau correspondence for the quintic threefold, and exhibits the first instance of the ``genus zero controls higher genus'' principle, in the sense of Givental's quantization formalism, for non-semisimple cohomological field theories.
\end{abstract}

\section{Introduction}

Over the last twenty-five years, there have been a number of important developments that have advanced our understanding of Gromov-Witten theory. Among these results, the genus-zero mirror theorems have provided closed formulas for the genus-zero Gromov-Witten potentials of a large number of target geometries \cite{GiventalMirror,LLY,Bertram,CFK2}, and Teleman's classification theorem for semisimple cohomological field theories \cite{Teleman} has led to explicit formulas for all-genus partition functions in terms of Givental's quantization formula \cite{GiventalHamiltonians,GiventalSemisimple}. One of the most important remaining open problems is to understand the all-genus partition functions of non-semisimple cohomological field theories, for which the Gromov-Witten theory of the quintic threefold $X:=V(W=x_0^5+\dots+x_4^5)\subseteq\bP^4$ is the prototypical example.

The Landau-Ginzburg/Calabi-Yau correspondence, which first arose in the study of string theory \cite{GVW,VW,Martinec}, suggests an equivalence between the Gromov-Witten theory of a Calabi-Yau hypersurface and the Landau-Ginzburg model of the defining equation of the hypersurface. The latter model is now mathematically understood in terms of Fan-Jarvis-Ruan-Witten invariants. In the case of the quintic threefold, Chiodo and Ruan proved that the genus-zero Gromov-Witten theory if $X$ can be identified with the genus-zero Fan-Jarvis-Ruan-Witten theory of the polynomial $W$ after analytic continuation and an explicit linear symplectic transformation $\U$ \cite{CR}. 

Motivated by Givental's quantization formula, Chiodo and Ruan suggested that the geometric quantization $\widehat{\U}$, which is an explicit differential operator constructed from $\U$, should identify the higher-genus Gromov-Witten and Fan-Jarvis-Ruan-Witten partition functions after analytic continuation. The genus-zero restriction of their quantization conjecture follows from the fact that $\U$ identifies the genus-zero theories. The main result of this work is the genus-one verification of Chiodo and Ruan's all-genus Landau-Ginzburg/Calabi-Yau conjecture.

\begin{mainresult}[Theorem \ref{thm:ellipticquantization}]
The genus-one potential determined by the action of $\widehat\U$ on the Fan-Jarvis-Ruan-Witten partition function of $W=x_0^5+\dots+x_4^5$ is equal to the analytic continuation of the genus-one Gromov-Witten potential of the quintic threefold $X=V(W)$.
\end{mainresult}

The theorem is significant for several reasons. First of all, it provides the first evidence for the higher-genus Landau-Ginzburg/Calabi-Yau correspondence. It has been suggested that the Landau-Ginzburg model could be instrumental in computing higher-genus Gromov-Witten invariants of the quintic threefold, and this theorem provides validity to that approach. Secondly, the theorem gives evidence for a general ``genus zero controls higher genus'' principle, in the sense of Givental, in which a correspondence between all-genus partition functions is determined by a genus-zero correspondence through an explicit quantization procedure. While such a principle has been studied extensively and proved in many cases for semisimple cohomological field theories, for example \cite{Teleman, GiventalSemisimple, BCR, Zong, CoatesIritani, HLSW, IMR}, this is the first significant evidence for such a principle in the non-semisimple case.

\subsection{Plan of the Paper}

We begin in Section \ref{sec:recap} by recalling the basic definitions in Gromov-Witten and Fan-Jarvis-Ruan-Witten theory. We recall some previously known results, including the genus-zero mirror theorems, the genus-zero Landau-Ginzburg/Calabi-Yau correspondence, and the genus-one mirror theorems. In Section \ref{sec:quantization}, we discuss the Birkhoff factorization of the symplectomorphism $\U$ and recall Givental's quantization formulas in order to make Theorem \ref{thm:ellipticquantization} precise. We also apply the string and dilaton equations to reduce the main theorem to the one parameter `small state-space'. In Section \ref{sec:genuszero}, we provide a proof of the genus-zero restriction of the quantization conjecture, mostly in order to set up notation for the genus-one correspondence. The proof of the genus-one correspondence occupies Sections \ref{sec:genusone}, \ref{sec:twisted}, and \ref{sec:loop2}, where we carefully analyze the vertex- and loop-type graphs that appear in the quantization formula. 

\subsection{Acknowledgements}

The authors are grateful to Y. Ruan for his support and encouragement. The second author would like to thank E. Clader for many valuable conversations related to this work. The first author is partially supported by the NSFC grants 11431001 and 11501013. The second author has been supported by the NSF postdoctoral research fellowship DMS-1401873.

\section{Recapitulation of Global Mirror Symmetry for the Quintic Threefold}\label{sec:recap}

In this section, we review the basic setup of Gromov-Witten and Fan-Jarvis-Ruan-Witten invariants, and we recall previously known mirror theorems concerning the genus-zero and genus-one invariants.

\subsection{Review of Gromov-Witten Theory}

Let $X$ denote the Fermat quintic threefold:
\[
X:=V(x_0^5+\cdots +x_4^5)\subset \mathbb{P}^4,
\]
and let $\M_{g,n}(X,d)$ denote the moduli space of $n$-pointed, genus-$g$, degree-$d$ stable maps to $X$. Gromov-Witten (GW) invariants of $X$ encode virtual intersection numbers
\begin{equation}\label{eq:gwcorrelators}
\left\langle\alpha_1\psi^{k_1}\cdots\alpha_n\psi^{k_n} \right\rangle^{CY}_{g,n,d}:=\int_{\left[ \M_{g,n}(X,d)\right]^{\vir}}\prod_{i=1}^n\ev_i^*(\alpha_i)\psi_i^{k_i}
\end{equation}
where $\alpha_i\in H^*(X,\bC)$, $\ev_i:\M_{g,n}(X,d)\rightarrow X$, is the $i$th evaluation map, $\psi_i$ is the descendent cotangent-line class, and $[-]^\vir$ is the virtual fundamental class. The correlators defined in \eqref{eq:gwcorrelators} are multilinear and symmetric. For the purposes of this paper, we focus on the \emph{ambient sector} $H^{CY}\subset H^*(X,\bC)$ of the state-space, obtained by restricting the cohomological insertions to the image of the restriction map: 
\[
H^{CY}:=\mathrm{Im}\left(H^*(\bP^4,\bC)\rightarrow H^*(X,\bC)\right).
\] 
A natural basis for $H^{CY}$ is given by $\{\varphi_0,\dots,\varphi_3\}$, where $\varphi_m$ is the pullback of $c_1(\cO(1))^m$ under the inclusion $X\hookrightarrow\bP^4$. The genus-$g$ GW potential is defined by
\[
F_g^{CY}(\bt):=\sum_{n,d}\frac{1}{n!}\left\langle\bt(\psi)^n \right\rangle^{CY}_{g,n,d},
\]
where 
\[
\bt(z)=\sum_{k\geq 0 \atop 0\leq m\leq 3} t_{k}^m\varphi_mz^k.
\]
We view the set of variables $\bt=\{t_k^m\}$ as formal parameters,\footnote{Typically, one introduces an additional Novikov parameter to keep track of the degree $d$. However, the divisor equation implies that the Novikov parameter and $t_0^1$ are redundant, allowing us to omit the Novikov parameter in our discussion.} and we write $\bt(z)$ when we want to emphasize the role of $z$. The sum is taken over all indices for which the underlying moduli space is nonempty. The GW partition function is defined by
\[
\D^{CY}(\bt,\hbar)=\exp\left(\sum_{g\geq 0}\hbar^{g-1} F_g^{CY}(\bt) \right).
\]

Following Givental \cite{GiventalSymplectic}, we define an infinite-dimensional vector space
\[
\cH^{CY}:=H^{CY}((z^{-1}))
\]
with symplectic form
\[
\Omega^{CY}(f(z),g(z))=\Res_{z=0}(f(z),g(-z))^{CY},
\]
where $(-,-)^{CY}$ denotes the Poincar\'e pairing on $X$. Let $(\bq,\bp)$ be the Darboux coordinates on $\cH^{CY}$ with respect to the basis $\varphi_mz^k$, so that a general element of $\cH^{CY}$ can be written
\[
\sum_{k\geq 0 \atop 0\leq m\leq 3} q_k^m\varphi_mz^k+\sum_{k\geq 0 \atop 0\leq m\leq 3}p^{m,k}\varphi^m(-z)^{-k-1},
\]
where $\varphi^m$ is Poincar\'e dual to $\varphi_m$. Viewing $F_0(\bt)$ as a formal function on $\cH_+^{CY}:=H^{CY}[z]$ via the dilaton shift:
\[
\bt(z)=\bq(z)+\varphi_0z,
\]
the genus-zero GW invariants are encoded in a Lagrangian subspace $\L^{CY}$, defined as the graph of the differential of $F^{CY}_0$:
\[
\L^{CY}:=\left\{p^{m,k}=\frac{\partial F^{CY}_0(\bt)}{\partial q_k^m} \right\}\subset \cH^{CY}.
\]
A general point of $\L^{CY}$ has the form
\[
J^{CY}(\bt,-z):=-z\varphi_0+\bt(z)+\sum_{n,d,m}\frac{1}{n!}\left\langle\bt(\psi)^n\;\frac{\varphi_m}{-z-\psi}\right\rangle_{0,n,d}^{CY}\varphi^m
\]
Givental proved that $\L^{CY}$ is a cone centered at the origin, and that every tangent space $T$ is tangent to $\L^{CY}$ exactly along $zT$. In particular, $\L^{CY}$ (and hence, the totality of genus-zero GW invariants) is determined by the finite-dimensional slice
\[
J^{CY}(t,-z)=-z\varphi_0 +t+\sum_{n,d,m}\frac{1}{n!}\left\langle t^n\;\frac{\varphi_m}{-z-\psi}\right\rangle_{0,n,d}^{CY}\varphi^m
\]
where $t=\sum_{0\leq m\leq 3}t^m\varphi_m$. The properties of the cone imply that
\[
\L^{CY}=\left\{\sum_r c_r(t,z)S^{CY}(t,z)^*(\varphi_r)\;:\; c_r(t,z)\in\cH_+^{CY} \right\},
\]
where\footnote{The asterisk in the notation refers to the fact that $S^{CY}(t,z)^*$ is the adjoint of a fundamental solution of the Dubrovin connection.}
\[
S^{CY}(t,z)^*(\varphi_r)=\frac{\partial J^{CY}(t,-z)}{\partial t^r}=\varphi_r+\sum_{n,d,m}\frac{1}{n!}\left\langle\phi_r\; t^n\;\frac{\varphi_m}{-z-\psi}\right\rangle^{CY}_{0,n,d}\varphi^m.
\]

In the particular case of the quintic threefold $X$, even more is true. It follows from dimension arguments along with the string and dilaton equations that $\L^{CY}$ is, in fact, determined by the one-dimensional slice along the \emph{small state-space} $t=\tau\varphi_1$:
\[
J^{CY}(\tau\varphi_1,-z)=-z\varphi_0 +\tau\varphi_1+\sum_{n,d,m}\frac{1}{n!}\left\langle (\tau\varphi_1)^n\;\frac{\varphi_m}{-z-\psi}\right\rangle_{0,n+1,d}^{CY}\varphi^m.
\]
By a slight abuse of notation, we often drop $\varphi_1$ in the notation when we restrict to the small state-space: $J^{CY}(\tau,-z):=J^{CY}(\tau\varphi_1,-z)$.

\subsection{Review of Fan-Jarvis-Ruan-Witten Theory}

Let $\M_{g,\vec m}^{1/5}$ denote the moduli space of stable $5$-spin curves with $n$ orbifold marked points having multiplicities $\vec m = (m_1,\dots,m_n)$. More precisely, a point in $\M_{g,\vec m}^{1/5}$ parametrizes a tuple $(C,q_1,\dots,q_n,L,\kappa)$ where
\begin{itemize}
\item $(C,q_1,\dots,q_n)$ is a stable orbifold curve with $\mu_5$ orbifold structure at all marks and nodes;
\item $L$ is an orbifold line bundle on $C$ and the $\mu_5$-representation $L|_{q_i}$ is multiplication by $\re^{2\pi\ri m_i/5}$;
\item $\kappa$ is an isomorphism
\[
\kappa:L^{\otimes 5}\cong\omega_{C,\log}.
\]
\end{itemize}

The (narrow) Fan-Jarvis-Ruan-Witten (FJRW) invariants of the quintic threefold encode the intersection numbers
\begin{equation}\label{eq:lgcorrelators}
\left\langle\phi_{m_1}\psi^{k_1}\cdots\phi_{m_n}\psi^{k_n} \right\rangle_{g,n}^{LG}:=5^{2-2g}\int_{\left[ \M_{g,\vec m+\vec 1}^{1/5}\right]^{\vir}}\prod_{i=1}^n\psi_i^{k_i},
\end{equation}
where $\psi_i$ is the $i$th cotangent line class on the coarse curve, and $[-]^\vir$ is the fifth power of the Witten class associated to the quintic threefold\footnote{The sign convention we use for the Witten class agrees with the original construction of Fan--Jarvis--Ruan \cite{FJR}.}. By convention, the correlators \eqref{eq:lgcorrelators} vanish if $m_i=4$ for any $i$. We let $H^{LG}$ denote the \emph{narrow} state-space, which is the complex vector-space generated by the formal symbols $\phi_0,\dots,\phi_3$ and with a non-degenerate pairing defined by $(\phi_i,\phi_j)^{LG}=5\delta_{i+j= 3}$\footnote{This pairing is different from the standard pairing in FJRW theory that was defined in \cite{FJR}, but it is consistent with our previous work \cite{GR} and matches better with the pairing in GW theory}. 

Analogously to GW theory, we define formal generating series $F_g^{LG}(\bt)$ and $\D^{LG}(\bt)$, we define a vector-space $\cH^{LG}$ with symplectic form $\Omega^{LG}$, and we define a Lagrangian subspace $\L^{LG}\subset\cH^{LG}$ which is determined by the slice $J^{LG}(t,-z)$ via the derivatives $S^{LG}(t,z)^*$. As in GW theory, the totality of genus-zero FJRW invariants are determined by the one-dimensional slice
\[
J^{LG}(\tau,-z)=-z\phi_0 +\tau\phi_1+\sum_{n,d,m}\frac{1}{n!}\left\langle (\tau\phi_1)^n\;\frac{\phi_m}{-z-\psi}\right\rangle_{0,n,d}^{LG}\phi^m.
\]

\subsection{Genus-Zero Mirror Theorems and the Landau-Ginzburg/Calabi-Yau Correspondence}

Define $I$-functions $I^{CY}(\fq,z)\in\cH^{CY}$ and $I^{LG}(\ft,z)\in\cH^{LG}$ by
\[
I^{CY}(\fq,z):=z\sum_{d\geq 0}\fq^{\varphi_1/z+d}\frac{\prod_{k=1}^{5d}(5\varphi_1+kz)}{\prod_{k=1}^d(\varphi_1+kz)^5},
\]
where $\varphi_1^k:=\varphi_k$ and $\varphi_4=0$, and
\[
I^{LG}(\ft,z):=z\sum_{a\geq 0}\frac{\ft^{a}}{z^aa!}\prod_{0< k<\frac{a+1}{5} \atop \langle k \rangle = \langle \frac{a+1}{5}\rangle}(kz)^5\phi_a,
\]
where $\phi_4=0$.\footnote{We warn the reader that the LG $I$-function defined here differs from the $I$-function defined in \cite{CR} by a factor of $\ft$. This keeps the notation consistent with our previous work \cite{GR}.}

The leading $z$-coefficients of the $I$-functions are especially important:
\[
I^{CY}(\fq,z)=:I_0^{CY}(\fq)\varphi_0z+I_1^{CY}(\fq)\varphi_1+\cO(z^{-1})
\]
and
\[
I^{LG}(\ft,z)=:I_0^{LG}(\ft)\phi_0z+I_1^{LG}(\ft)\phi_1+\cO(z^{-1}).
\]

The genus-zero mirror theorems, conjectured by Candelas--de la Ossa--Green--Parkes \cite{CdlOGP} in the GW setting and Huang--Klemm--Quackenbush \cite{HKQ} in the FJRW setting, provide an explicit solution to genus-zero GW and FJRW invariants in terms of the respective I-functions.

\begin{theorem}[Givental \cite{GiventalMirror}, Lian--Liu--Yau \cite{LLY}]\label{thm:cymirror}
Setting
\[
\tau^{CY}=\frac{I^{CY}_1(\fq)}{I^{CY}_0(\fq)},
\]
we have
\[
J^{CY}(\tau^{CY},z)=\frac{I^{CY}(\fq,z)}{I^{CY}_0(\fq)}.
\]
\end{theorem}

\begin{theorem}[Chiodo--Ruan \cite{CR}]\label{thm:lgmirror}
Setting
\[
\tau^{LG}=\frac{I^{LG}_1(\ft)}{I^{LG}_0(\ft)},
\]
we have
\[
J^{LG}(\tau^{LG},z)=\frac{I^{LG}(\ft,z)}{I^{LG}_0(\ft)}.
\]
\end{theorem}

Chiodo and Ruan also studied the relationship between the respective $I$-functions. They proved the following, which verifies the genus-zero Landau-Ginzburg/Calabi-Yau correspondence for the quintic threefold.

\begin{theorem}[Chiodo--Ruan \cite{CR}]\label{thm:CR}
Define a linear transformation $\U(-z):\cH^{LG}\rightarrow\cH^{CY}$ by
\[
\U(-z)(\phi_m)=\frac{\xi^{m+1}}{\re^{-2\pi\ri\varphi_1/z}-\xi^{m+1}}\frac{-2\pi\ri(-z)^m}{\Gamma(1+5\varphi_1/z)}\frac{\Gamma^5(1+\varphi_1/z)}{\Gamma^5(1-(m+1)/5)}.
\]
Then $\U(z)$ is symplectic and, upon identifying $\fq^{-1}=\ft^5$, there exists an analytic continuation of $I^{CY}(\fq,z)$ such that
\[
\U(z)(\ft I^{LG}(\ft,-z))=5\widetilde I^{CY}(\ft,-z).
\]
\end{theorem}

From the discussion above, it follows that Theorem \ref{thm:CR} can be rephrased as the statement that the symplectomorphism $\U(z)$ identifies Givental's Lagrangian cones upon analytic continuation. Following ideas of Givental, Chiodo and Ruan wrote the following in \cite{CR}:
\begin{displayquote}
...the quantization $\widehat\U$ is a differential operator which we expect to yield the full higher genus Gromov-Witten partition function when applied to the full higher genus Fan-Jarvis-Ruan-Witten partition function.
\end{displayquote}

In other words, Chiodo and Ruan conjectured that the higher-genus LG/CY correspondence can be formulated as an explicit relationship, depending only on genus-zero data, between the GW and FJRW partition functions. In Section \ref{sec:quantization} below, we make this conjecture more explicit, and we give a precise statement of our main result, which proves the genus-one part of their conjecture.

\subsection{Genus-One Mirror Theorems}

Our proof of the genus-one LG/CY correspondence relies on the genus-one mirror theorems. In GW theory, the genus-one mirror theorem was conjectured by Bershadsky--Cecotti--Ooguri--Vafa \cite{BCOV} and originally proved by Zinger \cite{Zinger} (by combining the results in Kim--Lho \cite{KimLho} and Ciocan-Fontanine--Kim \cite{CFK5}, there is also a new proof using quasimap techniques).

\begin{theorem}[Zinger \cite{Zinger}]
Setting
\[
\tau^{CY}=\frac{I^{CY}_1(\fq)}{I^{CY}_0(\fq)},
\]
we have
\[
F_1^{CY}(\tau^{CY})=\log \left(I_0^{CY}(\fq)^{-\frac{31}{3}}\fq^{-\frac{25}{12}}(1-5^{5}{\fq}  )^{-\frac{1}{12}} \left(\fq\frac{d\tau^{CY}}{d\fq}\right)^{-\frac{1}{2}} \right).
\]
\end{theorem}

In FJRW theory, the genus-one mirror theorem was conjectured by Huang--Klemm--Quackenbush \cite{HKQ} and proved by the authors \cite{GR}.

\begin{theorem}[Guo--Ross \cite{GR}]
Setting
\[
\tau^{LG}=\frac{I^{LG}_1(\ft)}{I^{LG}_0(\ft)},
\]
we have
\[
F_1^{LG}(\tau^{LG})= \log \left(I_0^{LG}(\ft)^{-\frac{31}{3}}(1-({\ft}/{5})^{5}  )^{-\frac{1}{12}} \left(\frac{d\tau^{LG}}{d\ft}\right)^{-\frac{1}{2}} \right).
\]
\end{theorem}

\section{Birkhoff Factorization and Geometric Quantization}\label{sec:quantization}

In order to make the higher-genus Landau-Ginzburg/Calabi-Yau correspondence more explicit, we write the linear transformation $\U(z)$ as a matrix in the bases $\{\phi_m\}$ and $\{\varphi_m\}$. Following Coates--Ruan \cite{CoatesRuan}, we consider the Birkhoff factorization of the matrix $\U(z)$:
\[
\U(z)=\U_-\U_0\U_+
\]
where $\U_-=1+\cO(z^{-1})$ is upper triangular, $\U_+=1+\cO(z)$ is lower triangular, and $\U_0$ is a diagonal matrix that is constant in $z$. By analogy with Givental \cite{GiventalHamiltonians}, we define
\[
S^{-1}(z)=\U_-(z)
\]
and
\[
R(z)=\U_0\U_+(z)\U_0^{-1},
\]
so that 
\[
\U(z)=S^{-1}(z)R(z)\U_0.
\]
We view $R$ and $S$ as linear automorphisms of $\cH^{CY}$, and $\U_0$ as a linear identification of $\cH^{LG}$ and $\cH^{CY}$. Since $\U$ is symplectic (i.e. $\U(z)\U(-z)^*=1$, where the asterisk denotes adjoint), it is not hard to see that $S$, $R$, and $\U_0$ are also symplectic:
\[
S(z)S(-z)^*=R(z)R(-z)^*=\U_0\U_0^*=1.
\]

Consider the geometric quantizations $\widehat R$, $\widehat S^{-1}$, and $\widehat\U_0$, defined, for example, in \cite{GiventalHamiltonians}. These are differential operators, which can be computed explicitly by the following result.

\begin{theorem}[Givental \cite{GiventalHamiltonians}]\label{thm:quantize}
Let $\bq(z)=q_k^m\varphi_mz^k$ be coordinates on $\cH_+^{CY}$. Given a partition function $\D(\bq)$ on $\cH_+^{CY}$, the quantized operators act as follows.
\begin{enumerate}
\item The quantization of $\U_0$ acts by
\[
\widehat\U_0\;\D(\bq)=\D(\U_0^{-1}\bq).
\]
\item The quantization of $S^{-1}$ acts by
\[
\widehat{S^{-1}}\;\D(\bq)=e^{W(\bq,\bq)/2\hbar}\D\left(\left[S\bq\right]_+\right)
\]
where $\left[S\bq\right]_+$ is the power series truncation of $S(z)\bq(z)$ and the quadratic form $W(\bq,\bq)=\sum_{k,l}(W_{kl}q_k,q_l)^{CY}$ is defined by
\[
\sum_{k,l\geq 0}\frac{W_{kl}}{w^kz^l}:=\frac{S(w)^*S(z)-1}{w^{-1}+z^{-1}}.
\]
\item The quantization of $R$ acts by
\[
\widehat R\;\D(\bq)=\left[e^{\frac{\hbar}{2} V\left(\frac{\partial}{\partial_{\bq}},\frac{\partial}{\partial_{\bq}}\right)}\D\right](R^{-1}\bq)
\]
where $R^{-1}\bq$ is the power series $R^{-1}(z)\bq(z)$ and the quadratic form $V=\sum_{k,l}(p_k,V_{kl}p_l)^{CY}$ is defined by
\[
V(w,z)=\sum_{k,l\geq 0}V_{kl}w^kz^l=\frac{1-R(-w)^*R(-z)}{w+z}.
\]
\end{enumerate}
\end{theorem}

When a partition function is written in the coordinates $\bt(z)$, we apply the formulas in Theorem \ref{thm:quantize} by first identifying $\bt(z)$ and $\bq(z)$ via the \emph{dilaton shift}:
\[
\bq(z)=\bt(z)-\Phi_0z,
\]
where $\Phi_0=\varphi_0$ or $\phi_0$ depending on the context. To simplify notation, we introduce the following convention:
\[
\oD(\bq)=\D(\bt),
\]
where $\bq$ and $\bt$ are related by the dilaton shift. It is important to notice that, even though we might start with a partition function that is a formal series centered at $\bt(z)=0$, the outcome of acting by the quantized operator may be divergent at $\bt(z)=0$.

The Chiodo--Ruan conjecture can be stated more explicitly in the following form.

\begin{conjecture}[Chiodo--Ruan \cite{CR}]\label{conj:CR}
There exists an analytic continuation of $\D^{CY}$ such that
\[
\widetilde{\D^{CY}(\bt)}\propto\widehat{S^{-1}}\widehat R\;\widehat \U_0\;\D^{LG}(\bt),
\]
where the symbol `$\propto$' denotes equivalence up to a scalar multiple.
\end{conjecture}

The main result of this paper is the following partial verification of Conjecture \ref{conj:CR}.

\begin{theorem}\label{thm:ellipticquantization}
Conjecture \ref{conj:CR} holds for the genus-zero and genus-one potentials. In other words, there exists an analytic continuation and a constant $C$ such that, for $g\leq 1$,
\[
[\hbar^{g-1}]\widetilde{\log\left(\D^{CY}(\bt)\right)}=[\hbar^{g-1}]\log\left(\widehat{S^{-1}}\widehat R\;\widehat \U_0\;\D^{LG}(\bt)\right)+\delta_{g,1}C.
\]
\end{theorem}

\begin{remark}
In order to interpret the analytic continuation, we consider both sides as formal power series in the variables $\{t_k^m:(k,m)\neq(0,1)\}$ with coefficients that are analytic in $t_0^1$, and we analytically continue coefficient-by-coefficient. Implicit in Conjecture \ref{conj:CR} is the claim that both sides are analytic in $t_0^1$. The question of whether genus-$g$ potentials are analytic is open in general. We verify the necessary convergence of genus-zero and genus-one potentials throughout the course of our arguments.
\end{remark}

%We prove Theorem \ref{thm:ellipticquantization} in two steps. We first prove in Sections \ref{sec:genuszero} - \ref{sec:loop2} that it holds upon restricting $\t(z)=\widetilde\tau^{CY}\phi_1$ (this is the restriction to the `small state-space'). In Section \ref{sec:stringdilaton}, we apply the string and dilaton equations to show that the general case follows from this restriction to the small state-space (the reduction to the small state-space works in every genus). Before we do any of that, we begin by describing more explicitly the actions of the quantized operators in terms of Feynman graph expansions.

\subsection{Quantized Operators, Potential Functions, and Graph Sums}\label{sec:graphsums}

In order to investigate Theorem \ref{thm:ellipticquantization}, we consider intermediate partition functions 
\begin{align*}
\oD^A(\bq)&:=\widehat\U_0\;\oD^{LG}(\bq)\\
\oD^B(\bq)&:=\widehat R\;\oD^A(\bq)\\
\oD^C(\bq)&:=\widehat{S^{-1}}\;\oD^B(\bq).
\end{align*}

Notice that $\oD^{LG}(\bq)$ is centered at $\bq(z)=-\phi_0z$ while $\oD^A(\bq)$, $\oD^B(\bq)$, and $\oD^C(\bq)$ are centered at $\bq(z)=-\U_0\phi_0z$, $\bq(z)=-R(z)\U_0\phi_0z$, and $\bq(z)=-\left[\U(z)\phi_0z\right]_+$, respectively. For each partition function, we can write
\[
\oD^\bullet(\bq)=:\re^{\sum_{g\geq 0}\hbar^{g-1}\overline{F}_g^\bullet(\bq)}.
\]
Theorem \ref{thm:quantize} implies that the $\U_0$-action is a change of variables:
\[
\oD^A(\bq)=\oD^{LG}(\U_0^{-1}\bq)\hspace{.5cm} \Longrightarrow \hspace{.5cm} \oF_g^A(\bq)=\oF_g^{LG}(\U_0^{-1}\bq).
\]
The $R$-action is more interesting. We have
\begin{align}\label{eq:quaddiff}
\nonumber \sum_{g\geq 0}\hbar^{g-1}\overline F_g^B(\bq)&=\log\left(\oD^B\left(\bq\right)\right)\\
&=\log\left(\left[e^{\frac{\hbar}{2} V\left(\frac{\partial}{\partial_{\bq}},\frac{\partial}{\partial_{\bq}}\right)}\oD^A\right]\left(R^{-1}\bq\right)\right).
\end{align}
The action of the exponential of the quadratic differential operator in \eqref{eq:quaddiff} has a Feynman graph expansion, and the logarithm outputs only the connected graphs. Let $\Gamma$ denote a connected graph consisting of vertices $V$, edges $E$, and legs $L$, with each vertex $v$ labeled by a genus $g_v$. For each $v$, let $\val(v)$ be the total number of legs and edges adjacent to $v$, define $g(\Gamma)=b_1(\Gamma)+\sum_v g_v$ where $b_1$ denotes the first Betti number of the graph, let $F=\{v,e\}$ denote the set of flags, and let $F_v$ and $L_v$ denote the flags and legs adjacent to a vertex $v$. We have
\[
\overline F_g^B(\bq)=\sum_{\Gamma\; :\; g(\Gamma)=g}\frac{1}{|\Aut(\Gamma)|}\underline\Contr(\Gamma)
\]
where
\[
\underline\Contr(\Gamma)=\Res_{z_f=0}\prod_v\underline\Contr(v)\prod_e\underline\Contr(e)
\]
with vertices contributing
\[
\underline\Contr(v)=\left(\sum_{m_f,k_f}\left(\prod_{f\in F_v}\frac{\varphi^{m_f}}{z^{k_f+1}}\otimes\frac{\partial}{\partial q_{k_f}^{m_f}}\right) \oF_g^A(\bq)\right)_{\bq(z)\rightarrow R^{-1}\bq(z)}
\]
contracted along the edges by pairing with the two-tensor
\[
\underline\Contr(e)= V(z_f,z_{f'})=\sum_{m,m'}V(z_f,z_{f'})_{m,m'}\varphi_m\otimes\varphi^{m'}.
\]
Including the $S$-action, we have
\[
\oF_g^C(\bq)=\delta_{g,0}\hbar^{-1}W(\bq,\bq)/2+\sum_{\Gamma\; :\; g(\Gamma)=g}\frac{1}{|\Aut(\Gamma)|}\Res_{z_f=0}\prod_v\Contr(v)\prod_e\Contr(e)
\]
where $\Contr(e)=\underline\Contr(e)$, but we replace the vertex contributions with
\[
\Contr(v)=\sum_{m_f}\left\langle \prod_{l\in L_v} \left(\overline{\bq}(\psi_l)+\varphi_0\psi_l\right)\prod_{f\in F_v}\frac{\U_0^{-1}\varphi_{m_f}}{z_f-\psi_f}\right\rangle^{LG}_{g_v,\val(v)}\bigotimes_f\varphi^{m_f},
\]
where
\[
\overline\bq(z):=\U_0^{-1}R^{-1}(z)[S(z)\bq(z)]_+.
\]

\subsection{String and dilaton equations.}

In this section, we show that the dilaton and string equations commute with quantization, allowing us to reduce Conjecture \ref{conj:CR} to the small state-space. 

The dilaton equation asserts that, for $\bullet=CY$ or $LG$ and $\Phi_m=\varphi_m$ or $\phi_m$, we have
\[
\left\langle \Phi_0\psi\;\Phi_{m_1}\psi^{k_1}\cdots\Phi_{m_n}\psi^{k_n} \right\rangle^\bullet_{g,n+1,(d)}=(2g-2+n)\left\langle \Phi_{m_1}\psi^{k_1}\cdots\Phi_{m_n}\psi^{k_n} \right\rangle^\bullet_{g,n,(d)},
\]
whenever the moduli space on the right-hand side exists. The string equation asserts that
\[
\left\langle \Phi_0\;\Phi_{m_1}\psi^{k_1}\cdots\Phi_{m_n}\psi^{k_n} \right\rangle^\bullet_{g,n+1,(d)}=\sum_{i=1}^n\left\langle \Phi_{m_1}\psi^{k_1}\cdots\Phi_{m_i}\psi^{k_i-1}\cdots\Phi_{m_n}\psi^{k_n} \right\rangle^\bullet_{g,n,(d)},
\]
whenever the moduli space on the right-hand side exists. We interpret $\psi^{-1}=0$. In addition, by a virtual dimension count, the correlator $\left\langle \Phi_{m_1}\psi^{k_1}\cdots\Phi_{m_n}\psi^{k_n} \right\rangle^\bullet_{g,n,(d)}$ vanishes unless $\sum m_i+\sum k_i=n$. Using this vanishing, it is not hard to see that $\D^\bullet(\bt)$ can be reconstructed from its restriction to $\bt(z)=t_0^1\Phi_1$ by the dilaton and string equations and the initial conditions
\begin{equation}\label{eq:initial}
\left\langle \Phi_a\;\Phi_b\;\Phi_0\right\rangle^\bullet_{0,3,(0)}=5\delta_{a+b,3} \hspace{1cm}\text{and}\hspace{1cm} \left\langle \Phi_0\psi \right\rangle_{1,1,(0)}=-\frac{25}{3}.
\end{equation}
It is useful to rephrase the string and dilaton equations as differential operators. In terms of total descendent potentials, the dilaton equation can be rewritten as
\begin{equation}\label{eq:dilaton}
\left(\sum_{m,k}q_k^m\frac{\partial}{\partial q_k^m}+2\hbar\frac{\partial}{\partial \hbar}-\frac{25}{3} \right)\D^\bullet(\bt)=0,
\end{equation}
and it is well known (see, for example, Example 1.3.3.2 in \cite{Coates}) that the string equation takes the form
\begin{equation}\label{eq:string}
\widehat{1/z}\;\D^\bullet(\bt)=0.
\end{equation}
Moreover, the equations \eqref{eq:dilaton} and \eqref{eq:string} take into account the initial conditions \eqref{eq:initial}, and thus determine $\D^\bullet(\bt)$ uniquely from its restriction to $\bt(z)=t_0^1\Phi_1$. The following compatibility is important in order to reduce Conjecture \ref{conj:CR} to the small state-space.

\begin{lemma}\label{prop:reduction}
The formal series $\widehat{S^{-1}}\widehat R\;\widehat \U_0\;\oD^{LG}(\bq)$ centered at $\bq(z)=-\left[\U(z)\phi_0z\right]_+$ satisfies the dilaton equation \eqref{eq:dilaton} and the string equation \eqref{eq:string}.
\end{lemma}

\begin{proof} We start with the dilaton equation. We must prove
\begin{equation}\label{eq:dilatonoperator}
\left(\sum_{m,k}q_k^m\frac{\partial}{\partial q_k^m}+2g-2 \right)\overline F_g^C(\bq)=\delta_{g,1}\frac{25}{3}.
\end{equation}
First of all, notice that the genus-zero shift $W(\bq,\bq)/2$ is annihilated by the operator in \eqref{eq:dilatonoperator}, simply because it is homogenous of degree $2$ in $\bq$. Next, notice that
\[
\sum_{m,k}q_k^m\frac{\partial}{\partial q_k^m}=\sum_{m,k}\overline q_k^m\frac{\partial}{\partial \overline q_k^m}.
\]
Therefore, by applying the dilaton equation for FJRW invariants to each vertex in the graph sum expression of $\overline F_g^C(\bq)$, along with fact that Euler characteristics add:
\[
\sum_{v\in \Gamma}(2-2g_v-|F_v|)=2-2g_\Gamma,
\]
we observe that
\[
\sum_{m,k}q_k^m\frac{\partial}{\partial q_k^m}\overline F_g^C(\bq)=\sum_{m,k}\overline q_k^m\frac{\partial}{\partial \overline q_k^m}\overline F_g^C(\bq)=(2-2g)\overline F_g^C(\bq)+\delta_{g,1}\frac{25}{3}.
\]
This proves \eqref{eq:dilatonoperator}.

We now verify the compatibility of the string equation. We must prove
\begin{equation}\label{eq:stringoperator}
\widehat{1/z}\widehat{S^{-1}}\widehat R\;\widehat \U_0\;\oD^{LG}(\bq)=0.
\end{equation}
By the string equation in FJRW theory, we know
\[
\widehat{1/z}\;\D^{LG}(\bt)=0.
\]
Therefore, it suffices to check that $\widehat{1/z}$ commutes with $\widehat{S^{-1}}\widehat R\;\widehat \U_0.$ Clearly, $1/z$ commutes with each of $S^{-1}$, $R$, and $\U_0$, but a little care must be taken because the quantization procedure is not an algebra homomorphism. However, by the formula for the cocycle given in Section 1.3.4 of \cite{Coates}, we see immediately that the cocycle vanishes when we commute $\widehat{1/z}$ with $\widehat{S^{-1}}$ and $\widehat \U_0$. Upon noticing that the linear-in-$z$ terms of $R$ are strictly above the diagonal, we also see from Example 1.3.4.1 in \cite{Coates} that the cocycle vanishes when we commute $\widehat{1/z}$ with $\widehat R$. This prove \eqref{eq:stringoperator}.
\end{proof}

Using the reconstruction by the dilaton and string equations, we can make the following reduction.

\begin{corollary}\label{cor:reduction}
In order to prove Conjecture \ref{conj:CR}, it suffices to prove the restriction
\begin{equation}\label{eq:reduction}
\widetilde F_g^{CY}(t_0^1) = F_g^C(t_0^1)+\delta_{g,1}C.
\end{equation}
\end{corollary}

The rest of this paper is devoted to proving \eqref{eq:reduction} in the case of genus-zero and genus-one potential functions. The analytic continuation in \eqref{eq:reduction} is described as follows. By the $g\leq 1$ mirror theorems for GW invariants, $F_g^{CY}(\tau^{CY})$ is an analytic function near $\fq=0$. In the course of our arguments below, we verify that $F_g^C(\widetilde\tau^{CY})$ is also an analytic function at $\ft=0$. The analytic continuation of $\tau^{CY}$ in this expression can be computed explicitly by Theorem \ref{thm:CR}. Thus, the analytic continuation occurring in \eqref{eq:reduction} occurs after substituting $t_0^1=\tau^{CY}$ and takes $F_g^{CY}(\tau^{CY})$ from $\fq=0$ to $\ft=0$ along the same path that identifies $I$-functions in Theorem \ref{thm:CR}.

\section{Genus-Zero Correspondence and Tail Series}\label{sec:genuszero}

Our goal in this section is to prove the genus-zero correspondence in Theorem \ref{thm:ellipticquantization} and to set up some notation for studying generating series of rational tails that appear in the Feynman graph expansions for $F_g^C$. We begin by recalling a few important points about genus-zero descendent invariants and semi-classical limits.

If $M(z)$ is a symplectomorphism such that $\widehat M \D^\bullet=\D^\star$, then a careful study of the genus-zero Feynman graphs (see, for example, Section 3.5 in \cite{CPS}) implies that
\[
M\mathcal{L}^\bullet=\mathcal{L}^\star,
\]
where, as in Section \ref{sec:recap}, the Lagrangian cone $\mathcal{L}$ is the differential of the genus-zero potential. In particular, by identifying the parts that have non-negative powers of $z$, this implies that
\begin{equation}\label{eq:Jfunctions}
\oJ^\star(\bq,-z)=M(z)\oJ^\bullet(M\cdot\bq,-z)
\end{equation}
where
\[
M\cdot\bq(z):=\left[M(z)^{-1}\overline J^\star(\bq,-z)\right]_+.
\]
Keep in mind that the change of variables $M\cdot\bq(z)$ shifts the center of the power series. The next result is a consequence of Theorem \ref{thm:CR}.%If $M$ is a symplectomorphism with only non-positive powers of $z$, then
%\[
%M\cdot\bq(z)=\left[M(z)^{-1}\bq(z)\right]_+
%\]
%is independent of the $J$-function.

\begin{proposition}\label{prop:Jfunctions}
Setting $\tau^C:=\widetilde \tau^{CY}$, we have
\[
\widetilde{J}^{CY}(\tau^C,z)=J^C(\tau^C,z).
\]
\end{proposition}

\begin{proof}
By \eqref{eq:Jfunctions}, we see that
\[
\U(z)J^{LG}(\tau^{LG},-z)=\overline J^{C}\left(\left[\U(z) J^{LG}(\tau^{LG},-z)\right]_+,-z\right).
\]
By Theorem \ref{thm:CR}, the left-hand side can be rewritten as
\[
\U(z)J^{LG}(\tau^{LG},-z)=\frac{5\widetilde I_0^{CY}(\ft)}{tI_0^{LG}(\ft)}\widetilde J^{CY}(\tau^{C},-z).
\]
On the other hand, we have
\begin{align*}
\overline J^{C}\left(\left[\U(z) J^{LG}(\tau^{LG},-z)\right]_+,-z\right)&= J^C\left(\frac{-5\widetilde I_0^{CY}(\ft)\varphi_0z+5\widetilde I_1^{CY}(\ft)\varphi_1}{tI_0^{LG}(\ft)}+\varphi_0z,-z\right)\\
&=:J^C\left(T_0\varphi_0z+T_1\varphi_1 ,-z\right),
\end{align*}
which is centered at $T_0=T_1=0$. Expanding as a Tayler series, we have
\[
J^C\left(T_0\varphi_0z+T_1\varphi_1 ,-z\right)=\sum_{i,j}\frac{A_{i,j}}{i!j!}T_0^iT_1^j,
\]
the dilaton equation \eqref{eq:dilatonoperator} implies that
\[
A_{i,j}=(i+j-2)A_{i-1,j}.
\]
Using the fact that, for $j\geq 2$,
\begin{equation}\label{eq:analyticidentity}
\sum_{m}{m+j-2 \choose m}T_0^m=\frac{1}{(1-T_0)^{j-1}},
\end{equation}
we see that
\[
J^C\left(T_0\varphi_0z+T_1\varphi_1 ,-z\right)=\frac{5\widetilde I_0^{CY}(\ft)}{tI_0^{LG}(\ft)}J^C\left(\tau^C,-z \right),
\]
concluding the proof.
\end{proof}

\begin{remark}
The identity \eqref{eq:analyticidentity} allows us to write $J^C$, which is a priori centered at $T_0=T_1=0$, as an analytic function at $\ft=0$.
\end{remark}

\begin{corollary}
We have the following genus-zero correspondence:
\[
\widetilde F_0^{CY}(\tau^C)=F^C(\tau^C).
\]
\end{corollary}

\begin{proof}
By Proposition \ref{prop:Jfunctions}
\[
\left(\frac{\partial F_0^C(\bq)}{\partial q_0^1}\right)_{\bq(z)=\tau^C\varphi_1}=\widetilde{\left(\frac{\partial F_0^{CY}(\bq)}{\partial q_0^1}\right)}_{\bq(z)=\tau^{CY}\varphi_1}=:J_2(\tau^C),
\]
\[
\left(\frac{\partial F_0^C(\bq)}{\partial q_1^0}\right)_{\bq(z)=\tau^C\varphi_1}=\widetilde{\left(\frac{\partial F_0^{CY}(\bq)}{\partial q_1^0}\right)}_{\bq(z)=\tau^{CY}\varphi_1}=:J_3(\tau^C),
\]
and all other partial derivatives vanish at $\bq(z)=\tau^C\varphi_1$. Thus, applying the dilaton equation, we have
\[
2F_0^C(\tau^C)=2\widetilde{F}_0^{CY}(\tau^C)=\tau^CJ_2(\tau^C)-J_3(\tau^C).
\]
\end{proof}

By Corollary \ref{cor:reduction}, this completes the proof of Theorem \ref{thm:ellipticquantization} for $g=0$.

\subsection{Tail Series}

A significant portion of our analysis of the action of $\widehat\U$ on $\D^{LG}(\bt)$ concerns packaging genus-zero tails in the Feynman graph expansions introduced in Section \ref{sec:graphsums}. More specifically, define
\[
T(\bq,z)=\overline\bq(z)+\left(\U_0^{-1}\;\underset{z_f=0}{\Res}\;V(z,z_f)\sum_{k,m}\frac{\varphi^m}{z_f^{k+1}}\frac{\partial \oF_0^{B}(\bq)}{\partial (R^{-1}\bq)_k^m}\right)_{\bq(z)\rightarrow[S(z)\bq(z)]_+}.
\]
Before continuing, let us briefly parse the definition of $T(\bq,z)$. First, since
\[
\oF_0^B(\bq)=\sum_{\Gamma\; :\; g(\Gamma)=0}\underline\Contr(\Gamma),
\]
we see that the partial derivatives
\[
\sum_{k,m}\frac{\varphi^m}{z_f^{k+1}}\frac{\partial \oF_0^{B}(\bq)}{\partial (R^{-1}\bq)_k^m}
\]
specify in each graph contribution a leg with a particular insertion on it. Contracting with $V(z,z_f)$ and taking the residue turns the specified leg into a specified edge. Finally, applying $\U_0^{-1}$ and specializing the variables $\bq(z)\rightarrow[S(z)\bq(z)]_+$, we see that $T(\bq,z)$ is the contribution of all possible genus-zero trees attaching to a specified vertex in the graph contribution for $\oF^C(\bq)$. Adding $\overline\bq(z)$ simply corresponds to the contribution of the degenerate tree. We call $T(\bq,z)$ the \emph{tail series}. The next lemma describes $T(\bq,z)$ explicitly. 

\begin{lemma}\label{lem:vertextail}
With notation as above, we have
\[
T(\bq,z)=\U(z)\cdot\bq(z).
\]
\end{lemma}

\begin{proof}
We compute directly:
\begin{align*}
\underset{z_f=0}{\Res}\;V(z,z_f)\sum_{k,m}\frac{\varphi^m}{z_f^{k+1}}\frac{\partial \oF_0^{B}(\bq)}{\partial (R^{-1}\bq)_k^m}&=\underset{z_f=0}{\Res}\;V(z,z_f)R(z_f)^*\sum_{k,m}\frac{\varphi^m}{z_f^{k+1}}\frac{\partial}{\partial {q_k^m}}F_0^B(\bt)\\
&=\underset{z_f=0}{\Res}\;\frac{R(z_f)^*-R(-z)^*}{z+z_f}\sum_{k,m}\frac{\varphi^m}{z_f^{k+1}}\frac{\partial \oF_0^B(\bq)}{\partial {q_k^m}}\\
&=\underset{z_f=0}{\Res}\;\frac{R(z_f)^*}{z+z_f}\sum_{k,m}\frac{\varphi^m}{z_f^{k+1}}\frac{\partial \oF_0^B(\bq)}{\partial {q_k^m}}\\
&=\underset{z_f=0}{\Res}\;\frac{R(z_f)^*}{z+z_f}\left(\oJ^B(\bq,z_f)-\bq(-z_f)\right).
\end{align*}
Therefore,
\begin{align*}
\underset{z_f=0}{\Res}\;V(z,z_f)\sum_{k,m}\frac{\varphi^m}{z_f^{k+1}}\frac{\partial \oF_0^{B}(\bq)}{\partial (R^{-1}\bq)_k^m}&=\Res_{z_f=0}\frac{R^{-1}(z_f)}{z_f-z}\left(\oJ^B(\bq,-z_f)-\bq(z_f)\right)\\
&=\left[R^{-1}(z)\oJ^B(\bq,-z) \right]_+-R^{-1}(z)\bq(z).
%&=R\cdot\bq(z)-R^{-1}(z)\bq(z).
\end{align*}
To obtain $T(\bq,z)$ from this, we multiply both sides by $\U_0^{-1}$, substitute $\bq(z)\rightarrow[S(z)\bq(z)]_+$, and add $\overline\bq(z)$, obtaining
\begin{align*}
T(\bq,z)%&=\U_0^{-1}\left(R^{-1}(z)\cdot[S(z)\bq(z)]_+\right)\\
&=\U_0^{-1}\left[R^{-1}(z)\oJ^B\left([S(z)\bq(z)]_+,-z\right)\right]_+\\
&=\U_0^{-1}\left[R^{-1}(z)S(z)\oJ^C\left(\bq(z),-z\right)\right]_+\\
&=\U(z)\cdot\bq(z).
\end{align*}
\end{proof}

\section{Genus-One Correspondence}\label{sec:genusone}

In regards to the genus-one potential, there are two types of graphs which appear: the \emph{vertex-type graphs} consist of trees with a unique genus-one vertex, and the \emph{loop-type graphs} consist of graphs $\Gamma$ with $b_1(\Gamma)=1$ and with $g_v=0$ for all $v\in V$. We separate the contributions from the two types of graphs, and we write
\[
F_1^C(\bt)=F_1^C(\bt)_V+F_1^C(\bt)_L.
\]
We now analyze these contributions.

\subsection{Vertex-Type Graphs}\label{sec:vertex}

By definition of the tail series, the contribution from the vertex-type graphs to $\oF_1^C$ is equal to
\[
\oF_1^C(\bq)_V=\oF_1^{LG}\left(T(\bq,z)\right).
\]
Restricting to the small state-space, we obtain the following result.

\begin{proposition}\label{prop:vertextail}
We have 
\[
F_1^C(\tau^{C})_V= {F_1^{LG}\left( \tau^{LG} \right)}+\frac{25}{3}\left(\log\left(  \ft{I_0^{LG}(\ft)}\right)-\log\left(5\widetilde  I_0^{CY}(\ft) \right)\right)
\]
where the variables are related by $\tau^C:=\widetilde \tau^{CY}=\frac{\widetilde I_1^{CY}(\ft)}{\widetilde I_0^{CY}(\ft)}$, and $\tau^{LG}=\frac{I_1^{LG}(\ft)}{I_0^{LG}(\ft)}$.
\end{proposition}

\begin{proof}
By Lemma \ref{lem:vertextail}, we have
\begin{align*}
\oF_1^C(\bq)_V&=\oF_1^{LG}\left(T(\bq,z) \right)\\
&=\oF_1^{LG}\left( \U(z)\cdot\bq(z) \right)\\
&=\oF_1^{LG}\left(\left[\U(z)^{-1} \oJ^{C}(\bq,-z)\right]_+ \right).
\end{align*}

Specializing $\bt= \widetilde{\tau}^{CY}$, the GW mirror theorem (Theorem \ref{thm:cymirror}) and the genus-zero LG/CY correspondence (Proposition \ref{prop:Jfunctions}) imply that
\begin{align*}
F_1^C(\tau^C)_V&=F_1^{LG}\left(\left[\U(z)^{-1} \widetilde{J}^{CY}(\tau^{C},-z)\right]_++z\phi_0 \right)\\
&=F_1^{LG}\left(\frac{- \ft I_0^{LG}(\ft)z\phi_0+ \ft I_1^{LG}(\ft)\phi_1}{5 \widetilde I_0^{CY}(\ft)}+z\phi_0 \right)\\
&=F_1^{LG}\left(\frac{I_1^{LG}(\ft)}{  I_0^{LG}(\ft)} \right)-\log\left(\frac{  \ft I_0^{LG}(\ft)}{5 \widetilde I_0^{CY}(\ft)} \right)\left\langle \psi_1\phi_0\right\rangle_{1,1}^{LG},
\end{align*}
where the final equality follows from the dilaton equation.
\end{proof}

\subsection{Loop-Type Graphs}\label{sec:loop}

In order to study the loop-type graph contributions to $F_1^C$, we consider the one-form $d\oF_1^C(\bq)_L$, which packages loop-type graph contributions with one specified leg. We break the loop at the vertex where the tree supporting the specified leg attaches and analyze the resulting genus-zero graph contributions. Define the two-tensors
\[
\oV^\bullet(\bq,w,z):=\sum_m\frac{\varphi_m\otimes\varphi^m}{w+z}+\sum_{m,m',k,k'}\frac{\varphi^m\otimes\varphi^{m'}}{w^{k+1}z^{k'+1}}\frac{\partial^2\oF^\bullet(\bq)}{\partial q_k^m\partial q_{k'}^{m'}}.
\]
The next lemma determines $d\oF_1^C(\bt)_L$ in terms of $\oV^\bullet(\bq,w,z)$.

\begin{lemma}\label{lem:looptail}
We have
\begin{align}
\nonumber d\oF_1^C(\bq)_L=\frac{1}{2}\Res_{w=0\atop z=0}\; \big(d\oV^{LG}\left(\U\cdot\bq,w,z\right),\U^{-1}(w)\otimes \U^{-1}(z) \oV^{C}(\bq,-w-z)  \big)^{LG},
\end{align}
where the pairing contracts along each factor of the two-tensors.
\end{lemma}

\begin{proof}
By arguing as in the proof of Lemma \ref{lem:vertextail}, we have
\begin{equation}\label{loop}
 d\oF_1^B(\bq)_L=\frac{1}{2}\Res_{w=0 \atop z=0}\;\left(d\oV^A(T(\bq,z),w,z),R^{-1}(w)\otimes R^{-1}(z) \oV^B(\bq,-w,-z)\right).
\end{equation}
Therefore, to obtain $d\oF_1^C(\bq)$, we must replace $\bq$ in \eqref{loop} with $S^{-1}\cdot\bq$. Using the facts that $\oJ^B(S^{-1}\cdot \bq,-z)=S(z)\oJ^C(\bq,-z)$ and that $\oV^\bullet(\bq,-w,-z)$ is obtained from $\oJ^\bullet(\bq,-z)$ by applying the operator
\[
\sum_{k,m}\frac{\varphi^m}{(-w)^{k+1}}\otimes\frac{\partial}{\partial{q_k^m}}=S^{-1}(w)\sum_{k,m}\frac{\varphi^m}{(-w)^{k+1}}\otimes\frac{\partial}{\partial{(S^{-1}\cdot\bq)_k^m}},
\]
we have
\[
\oV^B(S^{-1}\cdot\bq,-w,-z)=S(w)\otimes S(z)\;  \oV^C(\bq,-w-z).
\]
Therefore, the second term in the pairing in \eqref{loop} becomes
\[
R(w)^{-1}S(w)\otimes R(z)^{-1}S(z)\cdot \oV^C(\bq,-w-z).
\]
Similarly, the first term becomes
\[
d\oV^A(R(z)\cdot(S^{-1}\cdot\bq),w,z)=\U_0\otimes\U_0\; d\oV^{LG}\left(\U\cdot\bq,w,z\right).
\]
The Lemma then follows from the fact that $\U_0^*=\U_0^{-1}$.
\end{proof}

If we turn off the descendent parameters by setting $\bt=t$, then the string and WDVV equations (see, for example, \cite{Coates} Proposition 1.4.1) imply that
\[
V^{CY}(t,z,w)=\frac{\sum_m S^{CY}(t,w)^*(\varphi_m)\otimes S^{CY}(t,z)^*(\varphi^m)}{w+z}.
\]
and
\[
V^{LG}(t,z,w)=\frac{\sum_m S^{LG}(t,w)^*(\phi_m)\otimes S^{LG}(t,z)^*(\phi^m)}{w+z}.
\]
Therefore, by further specializing $t=\tau^C$ $(=\widetilde \tau^{CY})$, using the genus-zero correspondence of Theorem \ref{thm:ellipticquantization}, and applying the dilaton equation as in the proof of Proposition \ref{prop:vertextail}, the residue in Lemma \ref{lem:looptail} simplifies to the following.

\begin{lemma}\label{lem:looptail2}
We have
\begin{align}\label{eq:residue}
dF_1^C(\tau^C)_L=
&\nonumber \frac{1}{2}\Res_{w=0 \atop z=0}\;\bigg(d\frac{\sum_m \U(-w)S^{LG}(\tau^{LG},w)^*(\phi_m)\otimes \U(-z)S^{LG}(\tau^{LG},z)^*(\phi^m)}{w+z}, \\
& \hspace{5cm}\frac{\sum_m \widetilde S^{CY}(\tau^C,-w)^*(\varphi_m)\otimes (z)\widetilde S^{CY}(\tau^C,-z)^*(\varphi^m)}{-w-z}\bigg)^{CY}.
\end{align}
\end{lemma}

In order to further study the residue \eqref{eq:residue}, it will be useful to work in canonical bases for quantum products. Although the GW and FJRW invariants associated to the quintic threefold do not yield semi-simple Frobenius manifolds, they both admit twisted extensions in genus zero that do admit semi-simple Frobenius manifolds. In the next section, we recall and study the twisted extensions.

\section{Interlude on Twisted Invariants}\label{sec:twisted}

In this section, we describe semi-simple twisted theories that extend the genus-zero GW and FJRW invariants.

\subsection{Twisted GW and $5$-spin Invariants}

Twisted GW invariants associated to the quintic threefold take inputs from the extended state-space $\overline H^{CY}$ with basis $\varphi_0,\dots,\varphi_4$ where $\varphi_i=c_1(\cO(1))^m\in H^*(\bP^4,\bC)$. To define them, we consider the natural $(\bC^*)^5$-action on $\bP^4$:
\[
(\alpha_1\dots,\alpha_5)\cdot(z_1,\dots,z_5):=(\alpha_1z_1,\dots,\alpha_5z_5).
\]
There is an induced $(\bC^*)^5$-action on $\M_{g,n}(\bP^4,d)$ and a natural lift to $R\pi_*\L^{\otimes 5}$. Lifting the $\varphi_i$ to equivariant cohomology where $\prod(\varphi_i-\lambda_i)=0$, the twisted GW invariants are defined by
\[
\langle\varphi_{m_1}\psi^{a_1}\cdots\varphi_{m_n}\psi^{a_n}\rangle_{g,n,d}^{CY,\lambda}:=\int_{[\M_{g,n}(\bP^4,d)]^\vir}\left(\prod_{i=1}^n\ev_i^*(\varphi_{m_i})\psi_i^{a_i}\right)e_{(\bC^*)^5}(R\pi_*\L^{\otimes 5}),
\]
where $e_{(\bC^*)^5}(-)$ is the equivariant Euler class. These invariants take values in localized equivariant cohomology
\[
H_{\loc}^*(\B(\bC^*)^5,\bC)=\bC[\lambda_1^{\pm 1},\dots,\lambda_5^{\pm 1}].
\]
We recover the genus-zero GW invariants of the quintic by restricting the genus-zero twisted invariants to the ambient state-space $H^{CY}\subset\overline H^{CY}$ and taking the non-equivariant limit $\lambda_i=0$. We define the \emph{shifted twisted GW invariants} by
\[
\left\langle\left\langle\varphi_{m_1}\psi^{a_1}\cdots\varphi_{m_n}\psi^{a_n} \right\rangle\right\rangle^{CY,\lambda}_{g,n}(\tau):=\sum_d\sum_{k\geq 0}\frac{\tau^k}{k!}
\left\langle\varphi_{m_1}\psi^{a_1}\cdots\varphi_{m_n}\psi^{a_n}\; \varphi_1\cdots \varphi_1 \right\rangle^{CY,\lambda}_{g,n+k,d} .
\]

We are primarily interested in the specialization $\lambda_i=\xi^i\lambda$ where $\xi=\exp(2\pi\ri/5)$. Since the unspecialized correlators are symmetric in $\{\lambda_i\}$, the specialized correlators are Laurent polynomials in $\lambda^5$. The CY $I$-function can be extended to the (specialized) twisted setting:
\[
I^{CY,\lambda}(\fq,z):=z\varphi_0\sum_{d\geq 0}\fq^{\varphi_1/z+d}\frac{\prod_{k=1}^{5d}(5\varphi_1+kz)}{\prod_{k=1}^d\left((\varphi_1+kz)^5-\lambda^5\right)}
\]
where $\varphi_1^a:=\lambda^{5\left\lfloor\frac{a}{5}\right\rfloor}\varphi_a$.

Analogously, twisted $5$-spin invariants take inputs from the extended state-space $\overline H^{LG}$ with basis $\phi_0,\dots,\phi_4$. To define them, we consider the natural $(\bC^*)^5$-action on $L^{\oplus 5}$. This induces an action on $R\pi_*\L(-\Sigma_5)^{\oplus 5}$, where $\Sigma_5$ is the universal divisor of untwisted points. The twisted $5$-spin invariants are defined by
\[
\langle\phi_{m_1}\psi^{a_1}\cdots\phi_{m_n}\psi^{a_n} \rangle_{g,n}^{LG,\lambda} :=5^{2-2g}\int_{\left[\M_{g,\vec m+\vec 1}^{1/5}\right]}\left(\prod_{i=1}^n\psi_i^{a_i}\right)e_{(\bC^*)^5}((-R\pi_*\L(-\Sigma_5)^{\oplus 5})^\vee),
\]
taking values in
\[
H_{\loc}^*(\B(\bC^*)^5,\bC)=\bC[\lambda_1^{\pm 1},\dots,\lambda_5^{\pm 1}].
\]
We recover the genus-zero FJRW invariants associated to the quintic by restricting the genus-zero twisted invariants to the narrow state-space $H^{LG}\subset\overline H^{LG}$and taking the non-equivariant limit $\lambda_i=0$. We define the \emph{shifted twisted $5$-spin invariants} by
\[
\left\langle\left\langle\varphi_{m_1}\psi^{a_1}\cdots\varphi_{m_n}\psi^{a_n} \right\rangle\right\rangle^{LG,\lambda}_{g,n}(\tau):=\sum_{k\geq 0}\frac{\tau^k}{k!}
\left\langle\phi_{m_1}\psi^{a_1}\cdots\phi_{m_n}\psi^{a_n}\; \phi_1\cdots \phi_1 \right\rangle^{LG,\lambda}_{g,n+k} .
\]
As in the CY case, we are primarily interested in the specialization $\lambda_i=\xi^i\lambda$. The LG $I$-function can be extended to the (specialized) twisted setting:
\[
I^{LG,\lambda}(\ft,z)=z\sum_{a\geq 0}\frac{\ft^a}{z^aa!}\prod_{0< k<\frac{a+1}{5} \atop \langle k \rangle = \langle \frac{a+1}{5}\rangle}\left((kz)^5+\lambda^5\right)\phi_a.
\]

Notice that $I^{CY,\lambda}$ is annihilated by the Picard-Fuchs operator:
\begin{equation}\label{eq:cypf}
-\left(\fq\frac{d}{d\fq} \right)^5+\left(\frac{\lambda}{z}\right)^5+\fq\prod_{i=1}^5\left(5\fq\frac{d}{d\fq}+i\right)
\end{equation}
while $\ft I^{LG,\lambda}$ is annihilated by the Picard-Fuchs operator:
\begin{equation}\label{eq:lgpf}
\left( \frac{1}{5}\ft\frac{d}{d\ft}\right)^5 + \left(\frac{\lambda}{z}\right)^5 - \ft^{-5}\prod_{i=1}^5\left(\ft\frac{d}{d\ft}-1\right)
\end{equation}
Moreover, the differential operators \eqref{eq:cypf} and \eqref{eq:lgpf} agree upon setting $\fq^{-1}=\ft^5$.

\subsection{Genus-Zero Computations}

In what follows, we use $\Phi_m$ to denote $\varphi_m$ or $\phi_m$, depending on the context, and we use $x$ to denote $\fq$ or $\ft$. For $\bullet=CY$ or $LG$, we study the semi-simple Frobenius manifold on $\overline H^{\bullet}\otimes\bC[\lambda^{\pm 5}]$ where the pairing is defined by
\[
(\Phi_a,\Phi_b)^{\bullet,\lambda}:=\langle\langle\Phi_a\;\Phi_b\;\Phi_0 \rangle\rangle^{\bullet,\lambda}_{0,3}
\]
and the quantum product is defined by
\begin{equation}\label{eq:quantumproduct}
\Phi_a\star_\tau^\bullet\Phi_b:=\sum_{m}\langle\langle\Phi_a\;\Phi_b\;\Phi_m \rangle\rangle^{\bullet,\lambda}_{0,3}\Phi^m,
\end{equation}
where $\Phi^m$ is dual to $\Phi_m$ under the pairing.

For any $F(x,z)\in\bC[[x,z^{-1}]]$, define
\[
D^\bullet=\begin{cases}
\fq\frac{d}{d\fq} & \bullet=CY,\\
\frac{d}{d\ft} & \bullet=LG,
\end{cases}
\]
and define the Birkhoff factorization operator
\[
\mathbf M^\bullet F(x,z):=zD^\bullet\frac{F(x,z)}{F(x,\infty)},
\]
where, in the presence of state-space insertions, we set $\Phi_i=1$ in the denominator.

We inductively define series $I^{\bullet}_{p,q}(x)$ by
\begin{align}\label{eq:idecomp}
I^{\bullet}_{0,q}(x) := I^{\bullet,\lambda}_q(x),\quad \text{ and }\quad I^{\bullet,\lambda}_{p,q}(x) := D^\bullet \left( \frac{I^{\bullet}_{p-1,q}(x)}{I^{\bullet}_{p-1,p-1}(x)} \right) \, \text{ for $q\geq p>0$, }
\end{align}
so that
\[
(\mathbf  M^\bullet)^p  (I^{\bullet,\lambda}(x,z)/z) = \sum_{q\geq 0} I^{\bullet}_{p,p+q}(x) z^{-q} \Phi_{p+q}
\]
for $p\geq 0$.

We have the following expression of twisted $S$-operators in terms of $I$-functions.

\begin{proposition}\label{prop:si}
Define the twisted $S$-operators by
\[
S^{\bullet,\lambda}(x,z)^*(\Phi):=\Phi+\sum_m\langle\langle\Phi\;\frac{\Phi_m}{z-\psi} \rangle\rangle^{\bullet,\lambda}_{0,2}(x)\Phi^m.
\]
Then
\[
S^{\bullet,\lambda}(x,z)^*(\Phi_m):=\frac{(\mathbf M^\bullet)^m (I^{\bullet,\lambda}(x,z)/z)}{I^{\bullet}_{m,m}(x)}.
\]
\end{proposition}

\begin{proof}
This follows from standard properties of Givental's Lagrangian cone. See, for example, Lemma 7.4 in \cite{GR} for the proof in the LG setting.
\end{proof}

We have the following important properties of $I^{\bullet}_{p,p}$, which were proved in \cite{ZZ} for the CY case and in \cite{GR} in the LG case.

\begin{proposition}[\cite{ZZ} Theorem 2, \cite{GR} Lemma 7.6]
Define 
\[
L^\bullet=\begin{cases}
\left(1-\fq 5^5\right)^{-1/5} & \bullet=CY,\\
\left(1-(\ft/5)^5\right)^{-1/5} & \bullet=LG.
\end{cases}
\]
Then the following properties hold:
\begin{enumerate}
\item $I^{\bullet}_{0,0}\cdots I^{\bullet}_{4,4}=(L^\bullet)^5$;
\item $I^{\bullet}_{5+p,5+p}=\lambda^5I^{\bullet}_{p,p}$;
\item for $0\leq p\leq 4$, $I^\bullet_{p,p}=I^\bullet_{4-p,4-p}$.
\end{enumerate}
\end{proposition}

Following the arguments of \cite{GR}, we see that the quantum product \eqref{eq:quantumproduct} is semi-simple. In particular, define
\[
E_\alpha =\left\{\begin{array}{lr}\epsilon_\alpha & \bullet=CY\\ e_\alpha & \bullet=LG  \end{array}\right\} :=\frac{1}{5}\sum_i  \tilde \Phi_i \xi^{-i \alpha} ,\quad  \alpha = 0,1,2,3,4,
\]
where
\begin{align*}
\tilde \Phi_0= \Phi_0,\quad \tilde \Phi_1 = \frac{g^{-\frac{2}{5}} f^{-\frac{1}{5}}}{\lambda} \cdot \Phi_1,\quad &
\tilde \Phi_2 = \frac{g^{-\frac{4}{5}} f^{\frac{3}{5}}}{\lambda^{2}} \cdot \Phi_2,\\
\tilde \Phi_3 = \frac{g^{-\frac{6}{5}} f^{\frac{2}{5}}}{\lambda^{3}} \cdot  \Phi_3,\quad &
\tilde \Phi_4 = \frac{g^{-\frac{3}{5}} f^{\frac{1}{5}}}{\lambda^{4}} \cdot  \Phi_4.
\end{align*}
with $f:=I^\bullet_{2,2}/I^\bullet_{1,1}$ and $g=I^\bullet_{0,0}/I^\bullet_{1,1}$. Then
\[
E_\alpha\star_\tau^\bullet E_\beta=\delta_{\alpha,\beta}E_\alpha.
\]

Let $\{u^{\bullet,\alpha}\}$ be canonical coordinates, determined up to a constant by
\[
\sum_\alpha  E_\alpha d u^{\bullet,\alpha} =\Phi_1d\tau^\bullet.
\]

The next result computes the canonical coordinates explicitly in terms of a global one-form.

\begin{proposition}\label{prop:du}
We have
\[
d u^{\bullet,\alpha} = \xi^\alpha\lambda \cdot du,
\]
where $du$ is the global one-form
\[
du = L^{CY} \frac{d\fq}{\fq}=L^{LG}d\ft.
\]
\end{proposition}

\begin{proof}
The LG case is proved in Lemma 7.8 in \cite{GR}, and the CY case follows from the same arguments.
\end{proof}

We fix the constants of integration by declaring
\[
u^{CY,\alpha}=\xi^\alpha\lambda\log(\fq)+\cO(q)
\]
and
\[
u^{LG,\alpha}=\cO(\ft).
\]

The normalized canonical coordinates are defined by
\[
\tilde E_\alpha:=(\Delta_\alpha^\bullet)^{1/2} E_\alpha
\]
where
\[
\Delta_\alpha^\bullet=\frac{1}{\eta(E_\alpha,E_\alpha)^{\bullet,\lambda}}.
\]

We compute the pairing  on the canonical coordinates explicitly.

\begin{proposition}\label{lem:delta}
We have
\[
\Delta_\alpha^\bullet = (\xi^\alpha \lambda)^3 \frac{ (I^\bullet_{0,0})^2}{(L^\bullet)^2}
\]
\end{proposition}

\begin{proof}
The LG case is proved in Lemma 7.9 in \cite{GR}, and the CY case follows from the same arguments.
\end{proof}

The change of basis matrix between flat and normalized canonical coordinates is denoted by
\[
\Psi^\bullet_{\alpha m}:=(\widetilde E_\alpha,\Phi_m)^{\bullet,\lambda}.
\] 

From the above definitions, we can compute the change of basis explicitly.

\begin{proposition}\label{prop:cob}
We have 
\[
\Psi^{\bullet}_{\alpha m}=\xi^{\alpha (m-3/2)} c_{3-m}^{\bullet},
\]
where $c_i^\bullet$ satisfy
\begin{align*}
& c_{-1}^\bullet := \lambda^{\frac{5}{2}},  \quad c_0^\bullet := \lambda^{\frac{3}{2}}\frac{I_{0,0}^{\bullet}}{L^\bullet} , \quad c_1^\bullet := \lambda^{\frac{1}{2}}\frac{I_{0,0}^{\bullet}I_{1,1}^{\bullet}}{(L^\bullet)^2}\\
& c_2^\bullet := (c_1^\bullet)^{-1}=\lambda^{-\frac{1}{2}}\frac{I_{0,0}^{\bullet}I_{1,1}^{\bullet}I_{2,2}^{\bullet}}{(L^\bullet)^3}, \quad c_3^\bullet := (c_0^\bullet)^{-1}=\lambda^{-\frac{3}{2}}\frac{I_{0,0}^{\bullet}I_{1,1}^{\bullet}I_{2,2}^{\bullet}I_{3,3}^{\bullet}}{(L^\bullet)^4}
\end{align*}
\end{proposition}

For convenience, we also define $c_{4}^\bullet:=\lambda^{-\frac{5}{2}}$, so that $c_m^\bullet=(c_{3-m}^\bullet)^{-1}$ for $m=0,\dots,4$. The inverse matrix of $\Psi^\bullet$ is given by
\[
(\Psi^\bullet)^{-1}_{m \alpha}   =  \frac{\xi^{\alpha(3/2-m)}}{5} c_m^\bullet\\
\]

Since the quantum product is semi-simple for both types of twisted invariants, there is a canonical $R$-matrix that yields the higher-genus twisted invariants via Teleman's reconstruction theorem \cite{Teleman}. The diagonal entries of the linear term of the $R$-matrix can be computed explicitly.

\begin{proposition}
Define
\[
{(R_1^{CY})_{\alpha\alpha}}  =  \frac{1}{5} \frac{d}{du^\alpha} \left(\frac{5}{4}\log(L^{CY})-4 \log(I_0^{CY})-\log( \I_{1,1}^{CY}) -\frac{3}{4}\log(\fq)\right)
\]
and
\[
{(R_1^{LG})_{\alpha\alpha}}  =  \frac{1}{5} \frac{d}{du^\alpha} \left(\frac{5}{4}\log(L^{LG})-4 \log(I_0^{LG})-\log( \I_{1,1}^{LG}) \right).
\]
Then, up to constant terms, these matrices are equal to the linear terms of the canonical $R$-matrices associated to the respective semi-simple Frobenius manifolds.
\end{proposition}

\begin{proof}
In the LG case, this is Proposition 7.10 in \cite{GR}. In the CY case, the proof in \cite{GR} can be mimicked up to the point where
\[
(dR_1^{CY})_{\alpha\alpha}=\frac{1}{5\lambda \xi^\alpha} \left(\left( \frac{d \log  c_2^{CY}}{du}\right)^{2} + \left( \frac{d \log c_3^{CY}}{du}\right)^{2} \right) du.
\]
Letting $(-)'$ denote $\fq\frac{d}{d\fq}$, we then apply Lemma 3 of \cite{ZZ} at the second equality below to rewrite the right-hand side as
\begin{align*}
\left( \frac{d \log  c_2^{CY}}{du}\right)^{2} + \left( \frac{d \log c_3^{CY}}{du}\right)^{2}&=\frac{1}{(L^{CY})^2}\left(\frac{(c_2^{CY})'}{c_2^{CY}}+\frac{(c_3^{CY})'}{c_3^{CY}}\right)\\
&\hspace{0cm}=\frac{1}{L^{CY}}\left(-\frac{3}{4}(L^{CY})^4-\frac{1}{L^{CY}}\left(-5\frac{(L^{CY})'}{L^{CY}}+4\frac{(I_0^{CY})'}{I_0^{CY}}+\frac{(I_1^{CY})'}{I_1^{CY}} \right) \right)'\\
&\hspace{0cm}=\frac{d}{du}\left(\frac{1}{L^{CY}}\left(\frac{1}{4}(L^{CY})^5-1-4\frac{(I_0^{CY})'}{I_0^{CY}}-\frac{(I_1^{CY})'}{I_1^{CY}} \right)\right)\\
&\hspace{0cm}=\frac{d^2}{du^2}\left(\frac{5}{4}\frac{(L^{CY})'}{L^{CY}}-4\frac{(I_0^{CY})'}{I_0^{CY}}-\frac{(I_1^{CY})'}{I_1^{CY}}-\frac{3}{4}\log(\fq) \right).
\end{align*}
In the last equality, we have used the fact that $\frac{(L^{CY})'}{L^{CY}}=\frac{1}{5}\left((L^{CY})^5-1\right)$.
\end{proof}

Notice that the genus-one formulas can be obtained from the above $R$-matrices by the following formulas:
\begin{equation}\label{eq:dfcy}
dF_1^{CY}(\tau^{CY})=-\frac{200}{24}d\log(q^{1/5}I_0^{CY}(\fq))-\frac{5}{24}d\log(\fq^{1/5}L^{CY}(\fq))+\frac{1}{2}\sum_\alpha (R_1^{CY})_{\alpha\alpha}du^{\alpha}
\end{equation}
and
\begin{equation}\label{eq:dflg}
dF_1^{LG}(\tau^{LG})=-\frac{200}{24}d\log(I_0^{LG}(\ft))-\frac{5}{24}d\log(L^{LG}(\ft))+\frac{1}{2}\sum_\alpha (R_1^{LG})_{\alpha\alpha}du^{\alpha}
\end{equation}

\subsection{The Twisted Genus-Zero Correspondence}

We can extend Theorem \ref{thm:CR} to the twisted setting.

\begin{theorem}\label{thm:equivlgcy}
Define the linear transformation $\U^\lambda(-z):\overline \cH^{LG}\rightarrow\overline\cH^{CY}$ by
\[
\U^\lambda(-z)(\phi_m)=\frac{\xi^{m+1}}{\re^{-2\pi\ri\varphi_1/z}-\xi^{m+1}}\frac{-2\pi\ri(-z)^m}{\Gamma(1+5\varphi_1/z)}\prod_{i=0}^4\frac{\Gamma(1+\varphi_1/z-\xi^i\lambda/z)}{\Gamma(1-(m+1)/5-\xi^i\lambda/z)},
\]
where $\varphi_1^a:=\lambda^{5\left\lfloor\frac{a}{5}\right\rfloor}\varphi_a$. Then $\U^\lambda(-z)$ is a symplectic transformation and, upon identifying $\fq^{-1}=\ft^5$, there exists an analytic continuation of $I^{CY,\lambda}(\fq,z)$ such that
\[
\U^\lambda(-z)(\ft I^{LG,\lambda}(\ft,z))=5\widetilde I^{CY,\lambda}(\ft,z).
\]
\end{theorem}

\begin{proof}
The Mellin-Barnes method employed in \cite{CR} to prove Theorem \ref{thm:CR} (Corollary 4.2.4 in their paper) easily generalizes.
\end{proof}

Using Proposition \ref{prop:si} and Theorem \ref{thm:equivlgcy}, we can study the action of the symplectic transformation $\U^\lambda$ on the $S$-operators.

\begin{proposition}\label{prop:Mmatrix}
We have
\begin{equation}\label{eq:Mmatrixdef}
\U^\lambda(-z)S^{LG,\lambda}(\tau^{LG},z)^*= \widetilde S^{CY,\lambda}(\tau^C,z)^*M(\ft,z)
\end{equation}
where $M(\ft,z)$ has only non-negative powers of $z$. Moreover, $M(\ft,0)$ is diagonal with
\begin{equation}\label{eq:Mmatrixlinear}
M(\ft,0)=\left(\delta_{m,m'}(-1)^m\frac{5^{m+1}}{\ft^{m+1}}\frac{\widetilde I^{CY}_{0,0}\cdots\widetilde I^{CY}_{m,m}}{I^{LG}_{0,0}\cdots I^{LG}_{m,m}}\right)_{mm'}=-(\widetilde\Psi^{CY})^{-1}\Psi^{LG}
\end{equation}
and
\begin{equation}\label{eq:Mmatrix}
M(0,z)=R^\lambda(-z)\U^\lambda_0.
\end{equation}
\end{proposition}

\begin{proof}
The statement \eqref{eq:Mmatrixdef} follows from Theorem \ref{thm:equivlgcy} and general properties of Lagrangian cones. However, in order to prove \eqref{eq:Mmatrixlinear}, we first provide a more constructive proof of \eqref{eq:Mmatrixdef}. From Proposition \ref{prop:si} and Theorem \ref{thm:equivlgcy}, we compute
\begin{align*}
\widetilde S^{CY}(\tau^C,z)^*(\varphi_0)&=\frac{\widetilde I^{CY,\lambda}(\ft,z)}{\widetilde I^{CY}_{0,0}}=\frac{\ft}{5}\frac{I^{LG}_{0,0}}{\widetilde I^{CY}_{0,0}}\U^\lambda(-z)S^{LG,\lambda}(\tau^{LG},z)^*(\phi_0),\\
\widetilde S^{CY,\lambda}(\tau^C,z)^*(\varphi_1)&=\frac{z}{\widetilde I^{CY}_{1,1}}\left(-\frac{\ft}{5}\frac{d}{d\ft}\right)\widetilde S^{CY,\lambda}(\tau^C,z)^*(\varphi_0)\\
&=\cO(z)\U^\lambda(-z) S^{LG,\lambda}(\tau^{LG},z)^*(\phi_0)\\
&\hspace{.2cm}-\frac{\ft^2}{5^2}\frac{I^{LG}_{0,0}I^{LG}_{1,1}}{\widetilde I^{CY}_{0,0}\widetilde I^{CY}_{1,1}}\U^\lambda(-z)S^{LG,\lambda}(\tau^{LG},z)^*(\phi_1),\\
&\vdots\\
\widetilde S^{CY,\lambda}(\tau^C,z)^*(\varphi_4)&=\cO(z)\sum_{m=0}^3\U^\lambda(-z)\ft^{m+1}S^{LG,\lambda}(\tau^{LG},z)^*(\phi_m)\\
&\hspace{.2cm}+\frac{\ft^5}{5^5}\frac{I^{LG}_{0,0}I^{LG}_{1,1}I^{LG}_{2,2}I^{LG}_{3,3}I^{LG}_{4,4}}{\widetilde I^{CY}_{0,0}\widetilde I^{CY}_{1,1} \widetilde I^{CY}_{2,2} \widetilde I^{CY}_{3,3} \widetilde I^{CY}_{4,4}}\U^\lambda(-z)S^{LG,\lambda}(\tau^{LG},z)^*(\phi_4).
\end{align*}
The explicit formula for $M(\ft,0)$ in \eqref{eq:Mmatrixlinear} then follows from this computation and Proposition \ref{prop:cob}.

To prove \eqref{eq:Mmatrix}, multiply both sides of \eqref{eq:Mmatrixdef} by $S^\lambda(-z)$ to obtain
\[
R^\lambda(-z)\U_0^\lambda S^{LG,\lambda}(\tau^{LG},z)^*=S^\lambda(-z)\widetilde S^{CY,\lambda}(\tau^C,z)^*M(\ft,z).
\]
Since $S^{LG,\lambda}(\tau^{LG},z)^*\big|_{\ft=0}=1$, it suffices to prove that
\begin{equation}\label{eq:Mmatrix2}
S^\lambda(-z)\widetilde S^{CY,\lambda}(\tau^C,z)^*\big|_{\ft=0}=1.
\end{equation}
By Proposition \ref{prop:si} and the definition of the operator $\mathbf M^{CY}$, we compute, using the fact that $S^\lambda(-z)=1+\cO(z^{-1})$, that
\begin{align}\label{eq:SS}
\nonumber S^\lambda(-z)\widetilde S^{CY,\lambda}(\tau^C,z)^*(\varphi_k)&=S^\lambda(-z)\frac{\left(\mathbf M^{CY}\right)^k\left(\widetilde I^{CY,\lambda}(\ft,z)/z\right)}{[z^0\varphi_k]\left(\mathbf M^{CY}\right)^k\left(\widetilde I^{CY,\lambda}(\ft,z)/z\right)}\\
&=\frac{\left(\mathbf M^{CY}\right)^k\left(S^\lambda(-z)\widetilde I^{CY,\lambda}(\ft,z)/z\right)}{[z^0\varphi_k]\left(\mathbf M^{CY}\right)^k\left(S^\lambda(-z)\widetilde I^{CY,\lambda}(\ft,z)/z\right)}
\end{align}
Since $\U_0^\lambda$ is diagonal and $R^\lambda(-z)=1+\cO(z)$, we see that 
\begin{align}\label{eq:SS2}
\nonumber S^\lambda(-z)\widetilde I^{CY,\lambda}(\ft,z)/z&=\frac{\ft}{\sqrt{5}}R^\lambda(-z)\U_0^\lambda I^{LG,\lambda}(\ft,z)/z\\
&=\frac{\ft}{\sqrt{5}}\sum_{i=0}^4\left(\frac{(\U^\lambda_0)_{ii}\varphi_i}{i!}t^i+\cO(t^{i+1})\right)z^{-i}+\cO(t^6z^{-5})
\end{align}
Reinserting \eqref{eq:SS2} into \eqref{eq:SS}, one verifes \eqref{eq:Mmatrix2} from the definition of $\mathbf M^{CY}$.

\end{proof}

\section{Loop-Type Graphs Revisited}\label{sec:loop2}

We now return to the task of computing the residue in Equation \eqref{eq:residue}. We begin by reinterpreting the residue in terms of a non-equivariant limit of twisted invariants.

\begin{lemma}\label{lem:looplimit}
The loop-type contributions can be expressed as a non-equivariant limit:
\begin{align}\label{eq:eqresidue}
dF_1^C(\tau^C)_L=
&\nonumber \lim_{\lambda\rightarrow 0}\frac{1}{2}\Res_{w=0 \atop z=0}\;\bigg(d\frac{\sum_m \U^\lambda(-w)S^{LG,\lambda}(\tau^{LG},w)^*(\phi_m)\otimes \U^\lambda(-z)S^{LG,\lambda}(\tau^{LG},z)^*(\phi^m)}{w+z}, \\
& \hspace{5cm}\frac{\sum_m \widetilde S^{CY,\lambda}(\tau^C,-w)^*(\varphi_m)\otimes (z)\widetilde S^{CY,\lambda}(\tau^C,-z)^*(\varphi^m)}{-w-z}\bigg)^{CY}.
\end{align}
\end{lemma}

\begin{proof}
Recall that
\[
S^{\bullet,\lambda}(\tau,z)^*(\Phi)=\sum_m\left\langle\left\langle \Phi\;\frac{\Phi_m}{z-\psi} \right\rangle \right\rangle_{0,2}^{\bullet,\lambda}\Phi^m.
\]
Since
\[
\Phi^4 = \frac{1}{5\lambda^5}\Phi_4
\]
some care is required in taking the non-equivariant limit. It is not difficult to check that the LG correlators have a zero of order $5\times(\text{number of $\phi_4$ insertions})$ at $\lambda=0$. Along with the fact that $\U^\lambda(z)(\phi_4)$ has a zero of order $5$ at $\lambda=0$, we see that 
\[
\lim_{\lambda\rightarrow 0} \sum_m \U^\lambda(-w)S^{LG,\lambda}(\tau^{LG},w)^*(\phi_m)\otimes \U^\lambda(-z)S^{LG,\lambda}(\tau^{LG},z)^*(\phi^m)
\]
always exists, and it vanishes whenever there is a $\phi_4$ or $\phi^4$ insertion in either of the correlators. Similarly, it is not hard to see that the CY correlators have a zero of order $5$ at $\lambda=0$ whenever there is a $\varphi_4$ insertion, and therefore
\[
\lim_{\lambda\rightarrow 0} \left(\left(\sum_m \widetilde S^{CY,\lambda}(\tau^C,-w)^*(\varphi_m)\otimes\widetilde S^{CY,\lambda}(\tau^C,-z)^*(\varphi^m)\right)[\varphi^i\otimes\varphi^j]\right)
\]
always exists, and it vanishes whenever there is a $\varphi_4$ or $\varphi^4$ insertion in the correlators. Therefore, the limit of the pairing exists and vanishes whenever there is a $\Phi_4$ or $\Phi^4$ insertion in any of the correlators.
\end{proof}

Now that we understand the residue in terms of twisted $S$-matrices, we can rewrite it in terms of $R$-matrices, where the residue will become easy to compute. To do this, let $p_\alpha$ denote the equivariant cohomology class of the $\alpha$th $(\bC^*)^5$-fixed point of $\bP^4$, and set
\[
\tilde p_\alpha:=\frac{p_\alpha}{\sqrt{(p_\alpha,p_\alpha)^{CY}}}.
\]
The following result allows us to rewrite $S$-matrices in terms of $R$-matrices.

\begin{proposition}\label{prop:factorizations}
As a linear map from the basis $\{p_\alpha\}$ to the basis $\{\varphi_i\}$, the matrix series of the twisted CY fundamental solution 
\[
S^{CY,\lambda}(\tau^{CY},z)(\tilde p_\alpha)=\tilde p_\alpha+\sum_i\left\langle\left\langle\varphi_i\;\frac{\tilde p_\alpha}{z-\psi} \right\rangle\right\rangle_{0,2}^{CY,\lambda}\varphi^i
\]
factors canonically as
\[
(\Psi^{CY})^{-1}\underline R^{CY}(\fq,z)e^{U^{CY}/z}
\]
where $U^{CY}:=\diag(u^{CY,1},\dots,u^{CY,5})$ and $\underline R^{CY}(\fq,z)$ is an $R$-matrix of the Frobenius manifold associated to twisted GW invariants. In addition, the matrix series of the twisted $5$-spin fundamental solution
\[
S^{LG,\lambda}(\tau^{LG},z)\U^\lambda(-z)^*(\tilde p_\alpha)=\U^\lambda(-z)^*(\tilde p_\alpha)+\sum_i\left\langle\left\langle\phi_i\;\frac{\U^\lambda(-z)^*(\tilde p_\alpha)}{z-\psi} \right\rangle\right\rangle_{0,2}^{LG,\lambda}\phi^i
\]
factors canonically as
\[
-(\Psi^{LG})^{-1}\underline R^{LG}(\ft,z)e^{U^{CY}/z}
\]
where $\underline R^{LG}(\ft,z)$ is an $R$-matrix of the Frobenius manifold associated to the twisted $5$-spin invariants.

\end{proposition}

\begin{remark}
The reader should note that it does \emph{not} follow from Proposition \ref{prop:factorizations} that $\underline R^\bullet = R^\bullet$. However, since they are both $R$-matrices of the same semi-simple Frobenius manifold, they differ, at most, by right multiplication by a matrix of the form $\diag\left(\sum_{k\geq 0} a_{2i+1}z^{2k+1}\right)$. 
\end{remark}

\begin{proof}
The factorization of $S^{CY,\lambda}(\tau^{CY},z)$ was proved by Givental, \cite{GiventalSemisimple}, using materialization. A detailed proof can be found in Chapter 7 of \cite{LP}. To prove the factorization of $S^{LG,\lambda}(\tau^{LG},z)\U^\lambda(-z)^*$, we apply proposition \ref{prop:Mmatrix} to see that 
\begin{align*}
S^{LG,\lambda}(\tau^{LG},z)\U^\lambda(-z)^*(\tilde p_\alpha)&=M(\ft,z)^*\widetilde S^{CY,\lambda}(\tau^{CY},z)(\tilde p_\alpha)\\
&=M(\ft,z)^*(\widetilde \Psi^{CY})^{-1}\widetilde{\underline R}^{CY}(\ft,z)e^{U^{CY}/z}(\tilde p_\alpha)\\
&=-(\Psi^{LG})^{-1}\underline R^{CY}(\ft,z)e^{U^{LG}/z}(\tilde p_\alpha),
\end{align*}
where
\[
\underline R^{LG}(\ft,z):=-(\Psi^{LG})M(\ft,z)^*(\widetilde \Psi^{CY})^{-1}\widetilde{\underline R}^{CY}(\ft,z).
\]
To verify that $\underline R^{LG}(\ft,z)$ is an $R$-calibration of the Frobenius manifold, the following properties must be checked:
\begin{enumerate}
\item $U^{CY}$ is a diagonal matrix of canonical coordinates for the twisted $5$-spin Frobenius manifold,
\item $\underline R^{LG}(\ft,z)=1+\cO(z)$, and
\item $\underline R^{LG}(\ft,z)\underline R^{LG}(\ft,-z)^*=1$.
\end{enumerate}
The first property follows from Proposition \ref{prop:du}. The second follows from the fact that $\underline R^{CY}(\fq,z)=1+\cO(z)$, along with the second part of Proposition \ref{prop:Mmatrix} and the observation that $(\Psi^\bullet)(\Psi^\bullet)^*=1$. The third property follows from the observation that $(\Psi^\bullet)(\Psi^\bullet)^*=1$ and $M(\ft,z)M(\ft,-z)^*=1$.
\end{proof}

We now compute the residue.

\begin{proposition}\label{prop:residue}
We have
\[
dF_1^C(\tau^C)_L=\lim_{\lambda\rightarrow 0}\frac{1}{2}\sum_\alpha\left(\widetilde{\underline R}_1^{CY}(\ft)^*_{\alpha\alpha}-\underline R_1^{LG}(\ft)^*_{\alpha\alpha}  \right)du^\alpha,
\]
where $\underline R^\bullet_1(\ft)_{\alpha\alpha}$ denotes the diagonal entries of the linear-in-$z$ part of $\underline R^\bullet$.
\end{proposition}

\begin{proof}
Inserting the factorizations of Proposition \ref{prop:factorizations} into the residue in Lemma \ref{lem:looplimit}, we obtain
\begin{align*}
dF_1^C(\tau^C)_L=\frac{1}{2}\lim_{\lambda\rightarrow 0} \Res_{w=0 \atop z=0}\Bigg(& d\frac{\sum_\alpha \underline R^{LG}(\ft,w)^*(e_\alpha)\otimes \underline R^{LG}(\ft,z)^*(e_\alpha)}{w+z}\\
&\hspace{1cm}+ \frac{\sum_\alpha dU^{CY}\underline R^{LG}(\ft,w)^*(e_\alpha)\otimes \underline R^{LG}(\ft,z)^*(e_\alpha) }{w(w+z)}\\
&\hspace{1cm}+\frac{\sum_\alpha \underline R^{LG}(\ft,w)^*(e_\alpha)\otimes dU^{CY}\underline R^{LG}(\ft,z)^*(e_\alpha) }{(w+z)z},\\
&\hspace{2cm}\frac{\sum_\alpha \widetilde{\underline R}^{CY}(\ft,-w)^*(\epsilon_\alpha)\otimes \widetilde{\underline R}^{CY}(\ft,-z)^*(\epsilon_\alpha)}{-w-z}\Bigg)^{CY}.
\end{align*}
Expanding the denominators as Taylor series in either $\frac{z}{w}$ or $\frac{w}{z}$, the residue is easily computed to yield
\[
dF_1^C(\tau^C)_L=\frac{1}{2}\lim_{\lambda\rightarrow 0} \sum_\alpha\left(\widetilde{\underline R}_1^{CY}(\ft)^*_{\alpha\alpha}-\underline R_1^{LG}(\ft)^*_{\alpha\alpha} \right)du^\alpha.
\]
The proposition follows from the fact that $\underline R^\bullet_1(x)^*=\underline R_1^\bullet(x)$, which follows easily from the properties that $R^\bullet(x,z)=1+\cO(z)$ and $\underline R^\bullet(x,z)\underline R^\bullet(x,-z)^*=1$.
\end{proof}

\subsection{Comparing Constants}

Since both $R_1^\bullet$ and $\underline{R}_1^\bullet$ are linear terms of R-matrices, they differ by at most an additive constant. The purpose of this subsection is to compare the constants in order to obtain the following improvement of Proposition \ref{prop:residue}.

\begin{proposition}\label{prop:residue2}
We have
\[
dF_1^C(\tau^C)_L=\lim_{\lambda\rightarrow 0}\frac{1}{2}\sum_\alpha\left(\widetilde{R}_1^{CY}(\ft)_{\alpha\alpha}-R_1^{LG}(\ft)_{\alpha\alpha}  \right)du^\alpha.
\]
\end{proposition}

\begin{proof}
We prove Proposition \ref{prop:residue2} via a sequence of lemmas.

\begin{lemma}\label{lem:one}
We have
\[
\underline R^{CY}(\fq=0,z)=1,
\]
and
\[
\sum_\alpha \underline R^{CY}_1(\fq)_{\alpha\alpha}du^\alpha=\sum_\alpha R^{CY}_1(\fq)_{\alpha\alpha}du^\alpha+\frac{3}{4}du.
\]
\end{lemma}

\begin{proof}[Proof of Lemma]
By a standard application of the divisor equation and the localization isomorphism, we compute
\[
S^{CY,\lambda}(\tau^{CY},z)(\tilde p_\alpha)=\re^{\xi^\alpha\lambda\tau^{CY}/z}\tilde p_\alpha+\sum_{i \atop d>0}\re^{\tau^{CY}d}\left\langle\varphi_i\;\frac{\re^{\xi^\alpha\lambda\tau^{CY}/z}\tilde p_\alpha}{z-\psi}\right\rangle_{0,2,d}\varphi^i.
\]
Multiplying on the right by $\re^{-U^{CY}/z}$ and using the facts that $\tau^{CY}=\log(\fq)+\cO(\fq)$ and $u^{CY,\alpha}=\xi^{\alpha}\lambda\log(q)+\cO(q)$, we see that
\[
\left(S^{CY,\lambda}(\tau^{CY},z)\re^{-U^{CY}/z}\right)\bigg|_{\fq=0}(\tilde p_\alpha)=\tilde p_\alpha.\\
\]
Moreover, since $(\Psi^{CY})\big|_{\fq=0}$ is simply the change of basis from the $\{\varphi_i\}$ to $\{\tilde p_\alpha\}$, we see that, as a matrix,
\[
\underline R^{CY}(\fq=0,z)=\left(\Psi^{CY}S^{CY,\lambda}(\tau^{CY},z)\re^{-U^{CY}/z}\right)\bigg|_{\fq=0}=1.
\]
To prove the second statement, we note that $R^{CY}_1(\fq)$ and $\underline R^{CY}_1(\fq)$ are both linear terms of R-calibrations, so they only differ by an additive constant. Observing that 
\[
R^{CY}_1(\fq=0)_{\alpha\alpha}=-\frac{3}{20\lambda\xi^\alpha},
\] 
the computation above shows that
\begin{align*}
\sum_\alpha \underline R^{CY}_1(\fq)_{\alpha\alpha}du^\alpha&=\sum_\alpha \left(R^{CY}_1(\fq)_{\alpha\alpha}+\frac{3}{20\lambda\xi^\alpha}+1\right)du^\alpha\\
&=\left(\sum_\alpha \left(R^{CY}_1(\fq)_{\alpha\alpha}+\frac{3}{20\lambda\xi^\alpha}+1\right)\lambda\xi^\alpha\right) du\\
&=\sum_\alpha R^{CY}_1(\fq)_{\alpha\alpha}du^\alpha+\frac{3}{4}du.
\end{align*}
\end{proof}

\begin{lemma}\label{lem:two}
We have
\[
\widetilde{R}_1^{CY}(\ft=0)_{\alpha\alpha}=\frac{1}{5\lambda\xi^\alpha}\left(2\frac{\xi}{1+\xi}\frac{\Gamma^5\left(\frac{4}{5}\right)}{\Gamma^5\left(\frac{3}{5}\right)} +\frac{\xi(1+\xi)^3}{(1+\xi+\xi^2)^2}\frac{\Gamma^5\left(\frac{3}{5}\right)}{\Gamma^5\left(\frac{2}{5}\right)}\right)
\]
\end{lemma}

\begin{proof}[Proof of Lemma]
We know that
\[
\widetilde{R}_1^{CY}(\ft)_{\alpha\alpha}=\frac{1}{5\lambda\xi^\alpha}\frac{1}{L^{LG}}\frac{d}{d\ft}\left(-\frac{1}{4}\log(1-\ft^{-5}5^5)-4\log(\widetilde I_0^{CY}(\ft))-\log\left(-\frac{t}{5}\frac{d}{dt}\frac{\widetilde I_1^{CY}(\ft)}{\widetilde I_0^{CY}(\ft)}\right)+\frac{15}{4}\log(\ft) \right).
\]
Write
\[
\widetilde I_i^{CY}(\ft)=\frac{t}{5}\sum_{j=0}^3 b_{ij} I_j^{LG}(\ft)=\frac{t}{5}\left(b_{i0}+b_{i1}\ft+b_{i2}\frac{\ft^2}{2}+\cO(t^3) \right),
\]
where the coefficients $b_{ij}$ can be computed explicitly from the formula for $\U$ (see below, for example). Notice that the $\log(\ft)$ terms cancel and, disregarding constant terms in the derivative, we're left with
\begin{align*}
\widetilde{R}_1^{CY}(\ft)_{\alpha\alpha}=&\frac{1}{5\lambda\xi^\alpha}\frac{1}{L^{LG}}\frac{d}{d\ft}\bigg(-\frac{1}{4}\log(\ft^5-5^5)-4\log(b_{00}+b_{01}t+\dots)\\
&-\log\left(b_{00}(b_{00}b_{11}-b_{01}b_{10})+(b_{00}^2b_{12}-2b_{00}b_{01}b_{11}-b_{00}b_{02}b_{10}+2b_{01}^2b_{10})\ft+\dots \right)\bigg),
\end{align*}
where $+\dots$ denotes higher-order terms in $\ft$. Computing the derivative and setting $\ft=0$, we obtain
\[
\widetilde{R}_1^{CY}(\ft=0)_{\alpha\alpha}=\frac{1}{5\lambda\xi^\alpha}\left(\frac{b_{00}^2b_{12}+2b_{00}b_{01}b_{11}-b_{00}b_{02}b_{10}-2b_{01}^2b_{10}}{b_{00}(b_{01}b_{10}-b_{00}b_{11})} \right).
\]
Using the fact that, for $i=0,1$, 
\[
b_{ij}=\frac{(-1)^{j+1}(2\pi\ri)^{i+1}}{\Gamma^5\left(1-\frac{j+1}{5} \right)}\frac{\xi^{j+1}}{(1-\xi^{j+1})^{i+1}},
\]
the lemma follows by direct computation.
\end{proof}

\begin{lemma}\label{lem:three}
We have
\[
\lim_{\lambda\rightarrow 0}\sum_{\alpha}\lambda\xi^\alpha\underline  R_1^{LG}(\ft=0)_{\alpha\alpha}=\frac{3}{4}
\]
and
\[
\lim_{\lambda\rightarrow 0}\sum_\alpha \underline{R}^{LG}_1(\ft)_{\alpha\alpha}du^\alpha=\lim_{\lambda\rightarrow 0}\sum_{\alpha}R^{LG}_1(\ft)_{\alpha\alpha}du^\alpha+\frac{3}{4}du.
\]
\end{lemma}

\begin{proof}[Proof of Lemma]
Recall that
\[
\underline R^{LG}(\ft,z)=-(\Psi^{LG})M(\ft,z)^*(\widetilde \Psi^{CY})^{-1}\widetilde{\underline R}^{CY}(\ft,z), 
\]
so that
\[
\underline R_1^{LG}(\ft)_{\alpha\alpha}=-\left((\Psi^{LG})M_1(\ft)^*(\widetilde \Psi^{CY})^{-1}\right)_{\alpha\alpha}+\widetilde{\underline R}_1^{CY}(\ft)_{\alpha\alpha}.
\]
By the previous two lemmas, it suffices to prove that
\[
\lim_{\lambda\rightarrow 0}\sum_\alpha\lambda\xi^\alpha\left((\Psi^{LG})M_1(\ft)^*(\widetilde \Psi^{CY})^{-1}\right)_{\alpha\alpha}\bigg|_{\ft=0}=2\frac{\xi}{1+\xi}\frac{\Gamma^5\left(\frac{4}{5}\right)}{\Gamma^5\left(\frac{3}{5}\right)} +\frac{\xi(1+\xi)^3}{(1+\xi+\xi^2)^2}\frac{\Gamma^5\left(\frac{3}{5}\right)}{\Gamma^5\left(\frac{2}{5}\right)}.
\]
Let $\Psi_0^{LG}$ and $\widetilde\Psi_0^{CY}$ denote the specializations at $\ft=0$. Then by \eqref{eq:Mmatrixlinear} and \eqref{eq:Mmatrix}, we have
\begin{align*}
(\Psi^{LG}_0)M_1(0)^*(\Psi^{CY})^{-1}&=(\Psi_0^{LG})(R^\lambda(-z)\U_0^\lambda)^*_1(-\U_0^\lambda(\Psi_0^{LG})^{-1})\\
&=(\Psi_0^{LG})(\U_0^\lambda)^{-1}R^\lambda_1(\U_0^\lambda)(\Psi_0^{LG})^{-1}\\
&=(\Psi_0^{LG})(\U^\lambda_+)_1(\Psi_0^{LG})^{-1}.
\end{align*}
Therefore,
\begin{align*}
\sum_\alpha\lambda\xi^\alpha\left((\Psi_0^{LG})M_1(0)^*(\widetilde \Psi_0^{CY})^{-1}\right)_{\alpha\alpha}&=\mathrm{tr}\left(\diag(\lambda\xi^\alpha)(\Psi_0^{LG})(\U^\lambda_+)_1(\Psi_0^{LG})^{-1}\right)\\
&=\mathrm{tr}\left(\Lambda(\U^\lambda_+)_1 \right),
\end{align*}
where
\[
\Lambda:=(\Psi^{LG}_0)^{-1}\diag(\lambda\xi^\alpha)(\Psi_0^{LG})=\left(\begin{array}{ccccc}
0&0&0&0&\lambda^5\\
1&0&0&0&0\\
0&1&0&0&0\\
0&0&1&0&0\\
0&0&0&1&0
\end{array}
\right).
\]
Thus, the non-equivariant limit can be computed explicitly in terms of coefficients of $\U_+$:
\[
\lim_{\lambda\rightarrow 0}\sum_\alpha\lambda\xi^\alpha\left((\Psi_0^{LG})M_1(0)^*(\widetilde \Psi_0^{CY})^{-1}\right)_{\alpha\alpha}=(\U_+)_{01}+(\U_+)_{12}+(\U_+)_{23}.
\]
By expanding the gamma functions in terms of constants $C:=5/12$ and $E:=-40\zeta(3)/(2\pi\ri)^3$, Chiodo--Ruan \cite{CR} computed
\begin{equation}
\U(-z)=\begin{pmatrix}\
\frac{(-1)^{k+1}(2\pi \ri)}{\Gamma^5(1-\frac{k+1}{5})} \frac{\xi^{k+1}}{1-\xi^{k+1}} z^{k}& \\
\frac{(-1)^{k+1}(2\pi \ri)^2}{\Gamma^5(1-\frac{k+1}{5})} \frac{\xi^{k+1} }{(1-\xi^{k+1})^2} z^{k-1}& k=0,1,2,3\\
\frac{(-1)^{k+1}(2\pi \ri)^3}{\Gamma^5(1-\frac{k+1}{5})} \left(\frac{\xi^{k+1}(1+\xi^{k+1})}{2(1-\xi^{k+1})^3}+C\frac{\xi^{k+1}}{1-\xi^{k+1}}\right) z^{k-2}&
\\
\frac{(-1)^{k+1}(2\pi i)^4}{\Gamma^5(1-\frac{k+1}{5})} \left(\frac{\xi^{k+1}(1+4\xi^{k+1}+\xi^{2k+2})}{6(1-\xi^{k+1})^4}+C\frac{\xi^{k+1}}{(1-\xi^{k+1})^2}-E\frac{\xi^{k+1}}{1-\xi^{k+1}}\right)z^{k-3}&
\end{pmatrix}.
\end{equation}
By explicitly constructing $S(-z)$ in terms of elementary row operations, we have
\begin{align*}
S(-z)\U(-z)=&\left(\begin{matrix}
\frac{ -2\pi \ri   }{\Gamma^5(\frac{4}{5})} \frac{\xi}{1-\xi}  , &
\frac{ 2\pi \ri }{\Gamma^5(\frac{3}{5})} \frac{\xi^2}{1-\xi^2} z , &
\frac{  -2\pi \ri  }{\Gamma^5(\frac{2}{5})} \frac{\xi^3}{1-\xi^3} z^2 , &
\frac{ 2\pi \ri }{\Gamma^5(\frac{1}{5})}\frac{\xi^4}{1-\xi^4} z^3\\
  0&
\frac{ -(2\pi \ri)^2}{\Gamma^5(\frac{3}{5})} \frac{ \xi^3}{ (1-\xi^2)^2 }  , &
\frac{ (2\pi \ri)^2 }{\Gamma^5(\frac{2}{5})} \frac{\xi^4(1+\xi) }{(1-\xi^3)^2 } z  , &
\frac{- (2\pi \ri)^2 }{\Gamma^5(\frac{1}{5})} \frac{\xi^5 (1+\xi+\xi^2) }{ (1-\xi^4)^2 } z^2\\
  0&
 0&
\frac{ -(2\pi \ri)^3 }{\Gamma^5(\frac{2}{5})} \frac{\xi^6 }{ (1-\xi^3)^3 }    , &
\frac{ (2\pi \ri)^3 }{\Gamma^5(\frac{1}{5})} \frac{\xi^7 (1+\xi+\xi^2) }{ (1-\xi^4)^3 } z \\
  0&
  0&
 0&
\frac{ -(2\pi \ri)^4 }{\Gamma^5(\frac{1}{5})} \frac{\xi^{10}   }{ (1-\xi^4)^4 } \\
\end{matrix}\right),
\end{align*}
where, by definition, the right-hand side is $\U_0\U_+(-z)$. Therefore, we can compute
\[
(\U_+)_{01}=\frac{\xi}{1+\xi}\frac{\Gamma^5\left(\frac{4}{5}\right)}{\Gamma^5\left(\frac{3}{5}\right)},
\]
\[
(\U_+)_{12}=\frac{\xi(1+\xi)^3}{(1+\xi+\xi^2)^2}\frac{\Gamma^5\left(\frac{3}{5}\right)}{\Gamma^5\left(\frac{2}{5}\right)},
\]
and
\[
(\U_+)_{23}=\frac{\xi(1+\xi+\xi^2)(1-\xi^3)^3}{(1-\xi^4)^3}\frac{\Gamma^5\left(\frac{2}{5}\right)}{\Gamma^5\left(\frac{1}{5}\right)}.
\]
By applying the formula $\Gamma(x)\Gamma(1-x)=\pi/\sin(\pi x)$ and simplifying, it is straightforward to show that $(\U_+)_{23}=(\U_+)_{01}$, finishing the proof of the lemma.
\end{proof}

Proposition \ref{prop:residue2} now follows easily from Proposition \ref{prop:residue} and Lemmas \ref{lem:one} and \ref{lem:three}
\end{proof}

Combining Propositions \ref{prop:vertextail} and \ref{prop:residue2} with equations \eqref{eq:dfcy} and \eqref{eq:dflg}, we conclude that
\[
dF_1^C\left(\tau^C\right)= d\widetilde F_1^{CY}\left(\tau^C\right),
\]
which, by Corollary \ref{cor:reduction}, finishes the proof of Theorem \ref{thm:ellipticquantization}.

\bibliographystyle{alpha}
\bibliography{biblio}

\end{document}